\documentclass[a4paper, 10pt]{amsart}
\usepackage{mathtools, microtype, mathrsfs, xfrac, amssymb, enumerate, xspace}
\usepackage{amsmath,amsbsy}
\usepackage[alphabetic]{amsrefs}
\usepackage[latin1]{inputenc}
\usepackage{tikz-cd}
\usepackage[all,cmtip]{xy}
\usepackage{hyperref}
\hypersetup{
	colorlinks=true,
	linkcolor=blue,
	filecolor=magenta,      
	urlcolor=cyan,
}
\usepackage{faktor}
\usepackage{cleveref}
\usepackage{comment}
\usepackage{nicematrix}

\numberwithin{equation}{section}

\numberwithin{subsection}{section}

\allowdisplaybreaks[1]

\newenvironment{enumerate1}
{\begin{enumerate}[\upshape (1)]}
	{\end{enumerate}}

\newtheorem*{namedtheorem}{\theoremname}
\newcommand{\theoremname}{testing}

\newtheorem*{theorem-no-num}{Theorem}

\newtheorem{theorem}{Theorem}[section]
\newtheorem{proposition}[theorem]{Proposition}
\newtheorem{proposition-definition}[theorem]
{Proposition-Definition}
\newtheorem{corollary}[theorem]{Corollary}
\newtheorem{lemma}[theorem]{Lemma}

\theoremstyle{definition}
\newtheorem{definition}[theorem]{Definition}

\newtheorem{remark}[theorem]{Remark}

\theoremstyle{remark}

\renewcommand{\mathcal}{\mathscr}

 \newcommand\cB{\mathcal{B}}
 
\newcommand\cE{\mathcal{E}} \newcommand\cF{\mathcal{F}}
 
\newcommand\cI{\mathcal{I}} 
 \newcommand\cL{\mathcal{L}}
\newcommand\cM{\mathcal{M}} 
\newcommand\cO{\mathcal{O}} 
\newcommand\cQ{\mathcal{Q}} 
 \newcommand\cT{\mathcal{T}}
 \newcommand\cV{\mathcal{V}}
 
\newcommand\cY{\mathcal{Y}} 

\renewcommand\AA{\mathbb{A}}

\newcommand\GG{\mathbb{G}}

 \newcommand\PP{\mathbb{P}}
\newcommand\QQ{\mathbb{Q}}

 \newcommand\ZZ{\mathbb{Z}}

 \newcommand\bP{\mathbf{P}}

\newcommand\rmm{\mathrm{m}}

 \newcommand\bfd{\mathbf{d}}

 \newcommand\bfj{\mathbf{j}}
\newcommand\bfk{\mathbf{k}}

\newcommand*\bfell{\ensuremath{\boldsymbol\ell}}

\newcommand\arr{\ifinner\to\else\longrightarrow\fi}

\newcommand\arrto{\ifinner\mapsto\else\longmapsto\fi}

\newcommand{\hooklongrightarrow}{\lhook\joinrel\longrightarrow}

\newcommand{\eqdef}{\mathrel{\smash{\overset{\mathrm{\scriptscriptstyle def}} =}}}

\def\displaytimes_#1{\mathrel{\mathop{\times}\limits_{#1}}}

\def\displayotimes_#1{\mathrel{\mathop{\bigotimes}\limits_{#1}}}

\renewcommand\hom{\operatorname{Hom}}

\newcommand\Pic{\operatorname{Pic}}

\newcommand\spec{\operatorname{Spec}}

\newcommand\id{\mathrm{id}}

\newlength{\ignora}

\newcommand{\gm}{\GG_{\rmm}}

\newcommand{\GL}{\mathrm{GL}}
\newcommand{\SL}{\mathrm{SL}}
\newcommand{\PGL}{\mathrm{PGL}}
\newcommand{\Gr}{\mathrm{Gr}}

\DeclareFontFamily{U}{mathx}{\hyphenchar\font45}
\DeclareFontShape{U}{mathx}{m}{n}{
	<5> <6> <7> <8> <9> <10>
	<10.95> <12> <14.4> <17.28> <20.74> <24.88>
	mathx10
}{}
\DeclareSymbolFont{mathx}{U}{mathx}{m}{n}
\DeclareFontSubstitution{U}{mathx}{m}{n}
\DeclareMathAccent{\widecheck}{0}{mathx}{"71}
\DeclareMathAccent{\wideparen}{0}{mathx}{"75}

\renewcommand{\epsilon}{\varepsilon}

\newcommand{\Hilb}{\underline{\rm Hilb}}

\newcommand{\ch}[1][*]{\operatorname{CH}^{#1}}

\newcommand{\pr}{{\rm pr}}

\newcommand{\hilb}{\underline{\rm Hilb}}
\newcommand{\mpgl}[2]{\cM_{#1}^{\PGL}(#2)}
\newcommand{\mgl}[2]{\cM_{#1}^{\GL}(#2)}
\newcommand{\msl}[2]{\cM_{#1}^{\SL}(#2)}

\begin{document}
\title{Intersection theory on moduli of smooth complete intersections}
\author[A. Di Lorenzo]{Andrea Di Lorenzo}
	\address[A. Di Lorenzo]{Humboldt Universit\"{a}t zu Berlin, Germany}
	\email{andrea.dilorenzo@hu-berlin.de}
\maketitle
\begin{abstract}
    We provide a general method for computing rational Chow rings of moduli of smooth complete intersections. We specialize this result in different ways: to compute the integral Picard group of the associated stack ; to obtain an explicit presentation of rational Chow rings of moduli of smooth complete intersections of codimension two; to prove old and new results on moduli of smooth curves of genus $\leq 5$ and polarized K3 surfaces of degree $\leq 8$.
\end{abstract}
\section*{Introduction}
The investigation of rational Chow rings of moduli spaces, whose first instances can be traced back to the the work of Schubert on Grassmannians, is a domain that has been quite active in the last years. 

Among the most relevant results in this area, we have the determination of the Chow ring of $\overline{M}_3$, the moduli space of stable curves of genus three, by Faber (\cite{Fab-m3}), and the computation by several different authors of the Chow ring of $M_g$, the moduli space of smooth curves of genus $g$, for $2\leq g \leq 9$ (\cites{Fab-m4, Iza, PeV, CL}).

Let $0<r<n$ and let $\bfd=(d_1,\dots,d_r)$ be an increasing sequence of positive integers: in this paper, we study rational Chow rings of the moduli stacks $\cM^{\PGL}_n(\bfd)$ of smooth complete intersections of $r$ hypersurfaces of degree $d_1,\dots,d_r$ in $\PP^n$ (see \Cref{def:mpgl} for a rigorous definition). 

Our interest in the Chow ring of these moduli stacks stems from the fact that they can be used to gather information on the Chow ring of other moduli spaces, e.g. moduli of curves of low genus or moduli of polarized K3 surfaces of low degree (see \Cref{rmk:examples} for more on this). Moreover, stacks of complete intersections have already been the subject of some study, e.g. in the series of work by Benoist (\cites{benoist-thesis, Ben-deg, Ben-sep}) or in \cite{AI} when $\bfd=(2,2)$.

\subsection*{Main result}
The main technical result of this paper is the following Theorem, which concerns a stack denoted $\cM^{\GL}_n(\bfd)$ and from which all the statements on $\cM^{\PGL}_n(\bfd)$ are deduced. 

We are aware that at first sight this Theorem might not strike the reader as very explicit; for this reason, the remainder of the Introduction will be dedicated to explain its applications.
\begin{theorem-no-num}
We have
\[ \ch(\cM^{\GL}_n(\bfd))\simeq \QQ[c_1,c_2,\dots,c_{n+1},\gamma_1,\dots,\gamma_r]^{\mathfrak{S}_{\bfd}}/R.   \]
The ideal of relations $R$ is generated by cycles of the form
\[ \sum_{0\leq a_1,\dots,a_r \leq s} \gamma_1^{a_1}\cdots\gamma_r^{a_r}\cdot \pi_*\left( C_s(a_1,\dots,a_r)P(\beta_1,b_1,\dots,b_{s-1})\right) \]
where the coefficients in front of $\gamma_1^{a_1}\cdots\gamma_r^{a_r}$ are obtained via $\GL_{n+1}$-equivariant integration on a flag variety
of some specific cycles $C_s(a_1,\dots,a_r)P(\beta_1,b_1,\dots,b_{s-1})$.
Moreover, in degree $1$ the presentation above holds with $\ZZ$-coefficients.
\end{theorem-no-num}
The generators appearing above are certain symmetric functions in $\gamma_1,\dots,\gamma_r$, and the proof of this result is based on a vast generalization of a method introduced in \cite{FVis}.
A presentation for the Chow ring of $\cM^{\PGL}_n(\bfd)$ can then be obtained by simply adding the relation $c_1=0$. First quick applications of the Theorem are:
\begin{enumerate1}
\item the computation of \emph{the rational Chow ring of $M_5$} (\Cref{prop:chow M5}), already determined by Izadi: this computation is based on the fact that the stack of smooth, non-trigonal curves of genus five is isomorphic to $\cM^{\PGL}_4(2,2,2)$.
\item the computation of \emph{the rational Chow ring of an open subset of $K_8$, the moduli space of polarized K3 surfaces of degree eight} (\Cref{prop:chow U8}). This turns out to be trivial, hence all the non trivial cycles on $K_8$ of codimension $>0$ come from certain Noether-Lefschetz divisors.
\end{enumerate1}
The results above are obtained by applying localization formulas, implemented with Mathematica. Let us remark that once fixed $n$ and $\bfd$ the rational Chow ring of $\cM_n^{\PGL}(\bfd)$ can be explicitly worked out applying the same method.
\subsection*{Integral Picard groups and Benoist's formula}
Our Theorem can also be used to compute integral Picard groups. For instance, we prove the following.
\begin{theorem-no-num}
Suppose that the base field has characteristic $\neq 2$ or that $n$ is odd. Then:
\[ \Pic(\cM^{\PGL}_n(d,\dots,d)) \simeq \ZZ/N\ZZ,\quad N=\frac{\binom{n+1}{r}rd^{r}(d-1)^{n-r+1}}{{\rm mcm}(n+1,rd)}. \]
\end{theorem-no-num}
More generally, in \Cref{thm:pic pgl} we are able to determine the integral Picard group of $\cM^{\PGL}_n(\bfd)$ for every $\bfd=(d_1,\dots,d_r)$. Observe that the formula above, specialized to the case $\bfd=(2,2)$, recovers the main result of \cite{AI}.

A second application consists in the following: consider a product of projective spaces of the form $\PP H^0(\PP^n,\cO(d_1))\times\cdots\times \PP H^0(\PP^n,\cO(d_r))$; inside this variety there is a divisor whose points correspond to tuples of homogeneous forms $([f_1],\dots, [f_r])$ such that the projective scheme defined by the equations $f_1=f_2=\cdots=f_r=0$ is singular.

The multidegree of this divisor has been computed in \cite{Ben-deg} by Benoist using some toric geometry and results of Gelfand-Kapranov-Zelevinsky.
It turns out that the computation of this multidegree is equivalent to the computation of the integral Picard group of $\cM^{\SL}_n(\bfd)$. We do this in \Cref{thm:pic}, thus providing a different proof of Benoist's formula.

\begin{theorem-no-num}[\cite{Ben-deg}]
Suppose that the base field $k$ has characteristic $\neq 2$ or that $n$ is odd. Let
\[ d_1=\dots=d_{r_1}<d_{r_1+1}=\dots=d_{r_1+r_2}<\dots<d_{r_1+\cdots+r_{\ell-1}+1}=\dots=d_{r_1+\cdots+r_{\ell}}, \]
be positive integers and set $e_i=d_i-1$. Define $a_{j,1}=\gamma_{r_1+\cdots+r_{j-1}+1}+\cdots+\gamma_{r_1+\cdots+r_j}$. Then
\[ \Pic(\cM^{\SL}_n(\bfd))\simeq\langle a_{1,1},\dots,a_{\ell,1}\rangle/\langle F \rangle \]
where 
\[F= \sum_{i=1}^{r} \left( d_1d_2\cdots \widehat{d_i}\cdots d_r\sum_{j=1}^{r}\frac{1}{\prod_{j'\neq j} e_j-e_{j'}}\left(\frac{e_i^{n+1}-e_j^{n+1}}{e_i-e_j}\right)\right)\gamma_i \]
\end{theorem-no-num}
Our proof is based on Schubert calculus on a flag variety, combined with an interesting polynomial identity coming from the localization formula.
\subsection*{Complete intersections of codimension two}
From the main Theorem we are also able to derive a simple presentation of the Chow ring of $\cM^{\PGL}_n(d_1,d_2)$, the moduli stack of smooth complete intersections of codimension two. 
\begin{theorem-no-num}
Let $n\geq 3$ and $d_1>d_2\geq 2$ be integers such that the quantity (\ref{eq:det imp}) for $e_i=d_i-1$ is not zero. Then
\[\ch(\cM^{\PGL}_n(\bfd))\simeq \mathbb{Q}[\gamma_1]/(\gamma_1^2), \]
where $\gamma_1$ is a cycle of degree one.

If instead $d_1=d_2$ and the quantity (\ref{eq:det simple}) for $e=d_1-1$ is not zero, we have
\[\ch(\cM^{\PGL}_n(\bfd))\simeq \mathbb{Q}. \]
\end{theorem-no-num}
We give two direct applications of these results:
\begin{enumerate1}
\item in \Cref{cor:chow M4} we compute \emph{the rational Chow ring of $M_4$}, the moduli space of smooth curves of genus four; this ring has already been computed by Faber in \cite{Fab-m4}.
\item in \Cref{cor:chow U6} we compute \emph{the rational Chow ring of an open subset of $K_6$}, the moduli space of polarized K3 surfaces of degree six. The points in this subset correspond to polarized K3 surfaces whose polarization is very ample.
\end{enumerate1}

\subsection*{Outline of the paper}
In \Cref{sec:moduli} we define the stack $\cM^{\PGL}_n(\bfd)$ of complete intersections (\Cref{def:mpgl}) and we give a presentation of this stack as a quotient (\Cref{prop:presentation}). In the remainder of the Section we discuss the geometry of this stack.

In \Cref{sec:chow} we prove our main Theorem (\Cref{thm:chow}) and we specialize it to two interesting cases, namely to moduli of smooth curves of genus five and to moduli of polarized K3 surfaces of degree eight.

In \Cref{sec:pic} we compute the integral Picard group of $\cM^{\SL}_n(\bfd)$ (\Cref{thm:pic}) and $\cM^{\PGL}_n(\bfd)$ (\Cref{thm:pic pgl}).

In \Cref{sec:cod 2} we focus on smooth complete intersections of codimension two and we give a totally explicit presentation of the Chow ring of $\cM^{\PGL}_n(\bfd)$ in this case (\Cref{thm:chow cod 2 first} and \Cref{thm:chow cod 2 second}). We then apply these results to moduli of smooth curves of genus four and moduli of polarized K3 surfaces of degree six.

In \Cref{sec:quot} we gather a couple of useful results on quotient vector bundles and Grassmannians.

\subsection*{Notation and conventions}
All the schemes are schemes over a base field $k$. In most of the paper, we don't need any further assumption on the base field $k$. The only assumptions are the one stated for \Cref{thm:pic} and \Cref{thm:pic pgl}.

In the paper, the symbol $n\geq 2$ will always stand for the dimension of the projective space $\PP^n$. The integer $0<r<n$ will be the codimension of the complete intersections, and the degrees $d_1\leq \cdots \leq d_r$ will always be assumed to be $\geq 2$. The integers $r_1,\dots, r_\ell$ will be the ones such that $d_1=\dots=d_{r_1}<d_{r_1+1}=\dots=d_{r_1+r_2}<\dots<d_{r_1+\cdots+r_{\ell-1}+1}=\dots=d_{r_1+\cdots+r_{\ell}}.$ We will use $d'_i$ to indicate $d_{r_1+r_2+\cdots +r_i}$ and $e_i$ for the quantity $d_i-1$.

Every Chow ring is considered with $\QQ$-coefficients, unless otherwise stated.

\subsection*{Acknowledgments} 
We benefited from several conversations on this and related topics with Angelo Vistoli. For this, we thank him warmly. Part of this material is based upon work supported by the Swedish Research Council under grant no. 2016-06596 while the author was in residence at Institut Mittag-Leffler in Djursholm, Sweden during the fall of 2021.

\section{Moduli of complete intersections}\label{sec:moduli}
The main goal of this Section is to give a presentation of the stack $\cM_n^{\PGL}(\bfd)$ of (polarized) smooth complete intersections as a quotient stack (\Cref{prop:presentation}), presentation that will be used in the next Sections to perform intersection-theoretical computation.

We begin by recalling how $\cM^{\PGL}_n(\bfd)$ is defined (\Cref{def:mpgl}) and we list some examples of $\cM^{\PGL}_n(\bfd)$ for specific values of $\bfd$ that are of particular interest (see \Cref{rmk:examples}).

\Cref{prop:presentation} is proved by showing that a certain Hilbert scheme is isomorphic to a tower of Grassmannian bundles, that we define in \ref{sub:hilb}. The remainder of the Section is devoted to connect the equivariant Chow ring of this tower of Grassmannian bundles to the equivariant Chow ring of a much simpler object (\Cref{lemma:diagram}).
\subsection{The stack of complete intersections}
Let $k$ be a field. Fix two integers $n$ and $r$ with $0 < r < n$, and a sequence of positive integers $\bfd = (d_{1}, \dots, d_{r})$ with $d_i\leq d_{i+1}$ and $d_{i} \geq 2$ for all $i$. If $K$ is an extension of $k$, a closed subscheme $X \subseteq \PP^{n}_{K}$ is a complete intersection of type $\bfd$ if it has codimension $r$, and is the scheme theoretic intersection of $r$ hypersurfaces of degrees~$d_{1}$, \dots,~$d_{r}$. If $K'$ is an extension of $K$ and $X\subseteq \PP^{n}_{K'}$ is a closed subscheme, then $X_{K'} \subseteq \PP^{n}_{K'}$ a complete intersection of type $\bfd$ if and only if $X \subseteq \PP^{n}_{K}$ is \cite[Proposition~2.1.11]{benoist-thesis}.

We denote by $\Hilb_{\bfd,n}^{\rm sm}$ the subfunctor of the Hilbert scheme $\hilb_{\PP^{n,k}/k}$ such that if $S$ is a $k$-scheme, $\Hilb_{\bfd,n}^{\rm sm}(S) \subseteq \hilb_{\PP^n/k}(S)$ consists of closed subschemes $X \subseteq \PP^{n}_{S}$ that are finitely presented and flat over $S$, whose fibers are smooth complete intersections of type $\bfd$. This is a smooth open subscheme of $\hilb_{\PP^{n}/k}$ \cite[\S2.2.3]{benoist-thesis}.

There is a natural action of $\PGL_{n+1}$ over $\Hilb_{\bfd,n}^{\rm sm}$, coming from the action of $\PGL_{n+1}$ on $\PP^{n}$; if $R$ is a $k$-algebra, $A \in \PGL_{n+1}(R)$, and $X \subseteq \PP^{n}_{S}$ is in $\Hilb_{\bfd,n}^{\rm sm}(S)$, we define $A \cdot X$ to be the inverse image of $X$ under $A^{-1}\colon \PP^{n}_{S} \arr \PP^{n}_{S}$. 
\begin{definition}\label{def:mpgl}
   We set $\mpgl{n}{\bfd}:=[\Hilb_{\bfd,n}^{\rm sm}/ \PGL_{n+1}] $.
\end{definition}
   
Another way of interpreting $\mpgl{n}{\bfd}$ is as follows. If $P \arr S$ is a Brauer--Severi scheme of relative dimension $n$, a closed subscheme $X \subseteq P$ is a complete intersection of type $\bfd$ if for $S' \arr S$ a fully faithful finitely presented morphism, and $\PP^{n}_{S'} \simeq S'\times_{S}P$ an isomorphism of $S'$-schemes, the inverse image of $X$ in $\PP^{n}_{S'}$ is in $\Hilb_{\bfd,n}^{\rm sm}(S')$. If $T \arr S$ is a morphism and $X \subseteq P$ is a local complete intersection of type $\bfd$, the inverse image of $X$ in $T \times_{S} P$ is also a local complete intersection of type $\bfd$.

An object $\mpgl{n}{\bfd}(S)$, where $S$ is a $k$-scheme, is a pair $(P \arr S, X)$, where $P \arr S$ is a Brauer--Severi scheme of relative dimension $n$, and $X \subseteq P$ is a smooth complete intersection of type $\bfd$. The morphisms in $\mpgl{n}{\bfd}$ are the obvious ones.

The stack $\mpgl{n}{\bfd}$ was introduced by Benoist in \cite{benoist-thesis}; he determines, in particular, when $\mpgl{n}{\bfd}$ is a separated Deligne--Mumford stack, and when it has a quasi-projective moduli space.

\begin{remark}\label{rmk:examples}
$\mpgl{n}{\bfd}$ can be thought of as a stack of polarized algebraic varieties. In many cases the polarization is uniquely determined, and in this case $\mpgl{n}{\bfd}$ is in fact a stack of algebraic varieties, which in several cases is of considerable geometric interest.
\begin{enumerate1}

\item $\mpgl{2}{4}$ is the open subset of $\cM_{3}$ consisting of non-hyperelliptic curves of genus~$3$.

\item If $d \geq 4$, then it is well known that every smooth plane curve of degree $d$ has a unique linear $g^{2}_{d}$, (see for example \cite[Exercise~18, p.~56]{ACGH}). This means that the natural forgetful map $\mpgl{2}{d} \arr \cM_{g}$, where $\cM_{g}$ is the stack of smooth curves of genus $g \eqdef (d-1)(d-2)/2$, in injective on geometric points. One can show that this map is in fact a locally closed embedding.

\item $\mpgl{3}{2,3}$ is the stack of smooth non-hyperelliptic curves of genus $4$, while $\mpgl{4}{2,2,2}$ is the stack of curves of genus $5$ that are neither hyperelliptic nor trigonal (see the discussion in \cite[\S3]{DL-pic-curves}).

\item $\mpgl{3}{4}$ is the stack of K3 surfaces with a very ample polarization of degree~$4$, $\mpgl{4}{2,3}$ is the stack of K3 surfaces with a very ample polarization of degree~$6$, and $\mpgl{5}{2,2,2}$ is the stack of K3 surfaces with a very ample polarization of degree~$8$, which do not contain a curve of arithmetic genus $1$ and degree~$3$: see \cite[\S3]{DL-k3}.

\item $\mpgl{n}{2,2}$ is the stack of smooth complete intersections of two quadrics, which has been studied by Asgarli and Inchiostro in \cite{AI}.

\item If $n - r \geq 3$, then if $X \subseteq \PP^{n}_{K}$ is a complete intersection of type $\bfd$, the Picard group of $X$ is generated by the class of $\cO_{X}(1)$. Furthermore, a simple deformation-theoretic arguments reveals that a small deformation such a complete intersection is still a complete intersection of the same type. Using this, and the fact that $\dim_{K}H^{0}\bigl(X, \cO_{X}(1)\bigr) = n+1$ and $\dim_{K}H^{1}\bigl(X, \cO_{X}(1)\bigr) = 0$, it is an exercise to show that $\mpgl{n}{\bfd}$ is equivalent to the stack whose objects over a $k$-scheme $S$ are smooth proper morphisms $X \arr S$, whose geometric fibers are complete intersections dimension $n - r$ and type $\bfd$.

\item If $n - r = 2$ and $d_{1} + \dots + d_{r} \neq n + 1$, one can similarly conclude that $\mpgl{n}{\bfd}$ is equivalent to the stack whose objects over a $k$-scheme $S$ are smooth proper morphisms $X \arr S$, whose geometric fibers are complete intersection surfaces of type $\bfd$. The point is if $X \subseteq \PP^{n}_{K}$ is a smooth $2$-dimensional complete intersection of type $\bfd$, then by adjunction $\omega_{X/K} \simeq \cO_{X}(d_{1} + \dots + d_{r} - n-1)$. Since the Picard group of $X$ is torsion-free, this determines $\cO_{X}(1)$ uniquely. 

On the other hand, if $d_{1} + \dots + d_{r} \neq n + 1$ then $X$ is Calabi--Yau, and this will almost certainly fail for any possible value of $\bfd$.

\end{enumerate1}
\end{remark}

\begin{definition}
   We set $\mgl{n}{\bfd} :=[\Hilb_{\bfd,n}^{\rm sm}/ \GL_{n+1}]$, where the action of $\GL_{n+1}$ on $\Hilb_{\bfd,n}^{\rm sm}$ is induced by the projection $\GL_{n+1} \arr \PGL_{n+1}$.
   
   Similarly, we define $\msl{n}{\bfd} :=[\Hilb_{\bfd,n}^{\rm sm}/ \SL_{n+1}]$.
\end{definition}

These stacks can also be described in a spirit similar to the one above: an object of $\mgl{n}{\bfd}(S)$ can be thought of as a pair $(E, X)$, where $E$ is a locally free sheaf on $S$ of rank $n+1$, and $X \subseteq \PP(E)$ is a smooth complete intersection of type $\bfd$.

The stack $\mgl{n}{\bfd}$, while is not as as geometrically natural as $\mpgl{n}{\bfd}$, is used in many calculations of Picard groups and Chow rings of stacks of a geometric origin (see for example \cite{DL-k3, DL-pic-curves, DLFV, AI}).

The objects of $\msl{n}{\bfd}$ are pairs $(E,X,\varphi)$ where $E$ is a locally free sheaf on $S$ of rank $n+1$, the $S$-scheme $X\subset\PP(E)$ is a smooth complete intersection of type $\bfd$ and $\varphi:\det(E)\overset{\simeq}{\arr}\cO_S$ is an isomorphism.

\subsection{Hilbert schemes of smooth complete intersections}\label{sub:hilb}
As before, pick $n\geq 2$ and $0<r<n$ and let $\bfd=(d_1,d_2,\dots,d_r)$ be an $r$-uple of positive integers satisfying $d_1\leq d_2\leq\cdots\leq d_r$. There exists positive integers $r_1,\dots,r_\ell$ such that 
\begin{align*}
    & d_1=\cdots=d_{r_1},\\
    & d_{r_1+1}=\cdots=d_{r_1+r_2},\\
    & \vdots \\
    & d_{r_1+r_2+\cdots + r_{\ell-1}+1} = \cdots = d_{r_1+\cdots +r_\ell}.
\end{align*}
Define moreover $d'_{i}:=d_{r_1+\cdots + r_i}$, so that $d'_{1}<d'_{2}<\cdots < d'_{\ell}$. Obviously, the datum $(\{d'_{i}\},\{r_i\})$ is equivalent to the datum of an $r$-uple $(d_1,\dots,d_r)$.

Let $\cE_1:=H^0(\PP^n,\cO(d'_{1}))$ and let $\pi_1:\Gr(r_1,\cE_1)\to\spec{k}$ be the Grassmannian of $r_1$-planes in $\cE_1$. Over $\Gr(r_1,\cE_1)$ we have a tautological vector bundle $\cT_1\subset\pi_1^*\cE_1$. There is a natural evaluation map of sheaves over $\Gr(r_1,\cE_1)\times\PP^n$ given by $\pr_1^*\pi_1^*\cE_1\to \pr_2^*\cO(d'_{1})$. The image of the composition
\[\pr_1^*\cT_1\otimes\pr_2^*\cO(-d'_{1})\longrightarrow \pr_1^*\pi_1^*\cE_1\otimes\pr_2^*\cO(-d'_{1}) \longrightarrow \cO_{\Gr(r_1,\cE_1)\times\PP^n} \]
is an ideal, whose associated subscheme in $\Gr(r_1,\cE_1)\times\PP^n$ we denote $Y_1$. The fibers of $\pr_1:Y_1\to\Gr(r_1,\cE_1)$ are subschemes in $\PP^n$ of codimension $r_1$ defined by the vanishing of $r_1$ homogeneous polynomials of degree $d'_{1}$.

On $\Gr(r_1,\cE_1)$ we can consider the locally free sheaf $\cE_2:=\pr_{1*}(\pr_2^*\cO(d'_{2}))$, which we can use to define the Grassmannian bundle $\pi_2:\Gr(r_2,\cE_2)\to\Gr(r_1,\cE_1)$. With a slight abuse of notation, let us denote the closed subscheme $(\pi_2\times\id)^{-1}(Y_1)\subset\Gr(r_2,\cE_2)\times\PP^n$ as $Y_1$. Observe that
\begin{align*}
    \pi_2^*\cE_2 &= \pi_2^*\pr_{1*}(\cO_{Y_1}\otimes\pr_2^*\cO(d'_{2})) \\
    &\simeq \pr_{1*}((\pi_2\times\id)^*(\cO_{Y_1}\otimes\pr_2^*\cO(d'_{2})))\simeq\pr_{1*}(\cO_{Y_1}\otimes\pr_2^*\cO(d'_{2})).
\end{align*} 
If $\cT_2$ is the tautological bundle on $\Gr(r_2,\cE_2)$, we can construct the map
\[ \cT_2\otimes\pr_2^*\cO(-d'_{2}) \to \pi_2^*\cE_2\otimes\pr_2^*\cO(-d'_{2}) \simeq  \pr_{1*}(\cO_{Y_1}\otimes\pr_2^*\cO(d'_{2}))\otimes\pr_2^*\cO(-d'_{2}) \to \cO_{Y_1}\]
whose image is an ideal sheaf, whose associated closed subscheme is $Y_2\subset Y_1\subset\Gr(r_2,\cE_2)\times\PP^n$.
Repeating this process for every $d'_{i}$, we end up with a tower of Grassmannian bundles
\begin{equation}\label{eq:tower}
    \Gr(r_{\ell},\cE_{\ell}) \overset{\pi_{\ell}}{\to} \cdots \overset{\pi_3}{\to}\Gr(r_{2},\cE_{2}) \overset{\pi_2}{\to}\Gr(r_{1},\cE_{1})\overset{\pi_1}{\to} \spec{k}
\end{equation}
and a chain of closed subschemes $Y_{\ell}\subset \cdots \subset Y_{2}\subset Y_1\subset \Gr(r_{\ell},\cE_{\ell})\times \PP^n $.

Denote $\cY_{\bfd,n}:=Y_{\ell}$. By construction, the fibers of $\cY_{\bfd,n}\to \Gr(r_{\ell},\cE_{\ell})$ are complete intersections of type $\bfd=(d_1,\dots,d_r)$. 

We define $S_{\bfd,n}$ in $\Gr(r_{\ell},\cE_{\ell})$ as the singular locus of the map $Y_{\ell}\to \Gr(r_{\ell},\cE_{\ell})$. This is well known to be a closed subscheme, and the restriction of $\cY_{\bfd,n}$ over the complement of $S_{\bfd,n}$ is a family of smooth complete intersections, hence it defines a map
\begin{equation}\label{eq:map to hilb}
    \Gr(r_{\ell},\cE_{\ell})\smallsetminus S_{\bfd,n} \longrightarrow \Hilb_{\bfd,n}^{\rm sm}
\end{equation}
to the Hilbert scheme of smooth complete intersections in $\PP^n$ of type $\bfd$. Observe that the natural action of $\PGL_{n+1}$ on $\PP^n$ defines an action of the same group on $\Gr(r_{\ell},\cE_{\ell})$. It is easy to check that (\ref{eq:map to hilb}) is equivariant with respect to the $\PGL_{n+1}$-action on the Hilbert scheme. The same statement holds for the induced actions of $\GL_{n+1}$ and $\SL_{n+1}$.
\begin{proposition}\label{prop:presentation}
Let $G$ be either $\GL_{n+1}$, $\SL_{n+1}$ or $\PGL_{n+1}$. Then we have an isomorphism of $G$-schemes $\Hilb_{\bfd,n}^{\rm sm}\simeq \Gr(r_{\ell},\cE_{\ell})\smallsetminus S_{\bfd,n}$, hence a presentation
\[ \cM^G_n(\bfd) \simeq [\Gr(r_{\ell},\cE_{\ell})\smallsetminus S_{\bfd,n}/G] \]
\end{proposition}
\begin{proof}
We will construct an inverse to (\ref{eq:map to hilb}). Let $X\subset \PP^n_S \to S$ be a family of smooth complete intersections of type $\bfd$ and let $\cI$ be the ideal sheaf of $X$. We have an injective morphism of locally free sheaves
\begin{equation}\label{eq:inclusion 1}
    \pr_{1*}(\cI\otimes\pr_2^*\cO(d'_{1}))\hooklongrightarrow\pr_{1*}\pr_2^*\cO(d'_{1})\simeq H^0(\PP^n,\cO(d'_{1}))\otimes \cO_S.
\end{equation}
Observe that the sheaf on the left has rank $r_1$, hence it defines a map $S\overset{f_1}{\to} \Gr(r_1,\cE_1)$.

Let $X_1\subset \PP^n_S$ be the complete intersection of codimension $r_1$ defined by the homogeneous ideal associated to the image of (\ref{eq:inclusion 1}) and consider the inclusion
\[ \pr_{1*}(\cI\otimes\pr_2^*\cO(d'_{2})|_{X_1} )\hooklongrightarrow \pr_{1*}(\cO_{X_1}\otimes\pr_2^*\cO(d'_{2}) )\]
Observe that $X_1$ is the pullback of $Y_1\to\Gr(r_1,\cE_1)$ along $f_1:S\to\Gr(r_1,\cE_1)$, hence the sheaf above on the right is equal to $f_1^*\cE_2=f_1^*(\pr_{1*}(\cO_{Y_1}\otimes\pr_2^*\cO(d'_{2})))$. By hypothesis the locally free sheaf on the left has rank $r_2$, so we get a map $S\overset{f_2}{\to}\Gr(r_2,\cE_2)$.

Repeating this process, we eventually get a map $S\overset{f_{\ell}}{\to} \Gr(r_{\ell},\cE_{\ell})$ such that the pullback along this morphism of $Y_{\ell}\to\Gr(r_{\ell},\cE_{\ell})$ coincides with $X\to S$. In particular, the image of $f_{\ell}$ is contained in the complement of $S_{\bfd,n}$. Putting all together, we get a map
\[ \Hilb_{\bfd,n}^{\rm sm} \longrightarrow \Gr(r_{\ell},\cE_{\ell})\smallsetminus S_{\bfd,n} \]
which is an inverse to (\ref{eq:map to hilb}).
\end{proof}

\subsection{Another point of view}\label{subsubsec:another}
Let us give another and possibly more explicit construction of the tower of Grassmannians in (\ref{eq:tower}).
In what follows, we use the shorthand notation $H^0(\cO(d'_{i}))$ to denote $H^0(\PP^n,\cO(d'_{i}))$.
First define $\cE_{1,i}:=H^0(\cO(d'_{i}))$ and let $\pi_1:\Gr(r_1,\cE_{1,1})\to\spec{k}$ be the Grassmannian of $r_1$-planes in $H^0(\cO(d'_{1}))$.

Let $\cT_1\subset \pi_1^*\cE_{1,1}$ be the tautological bundle over $\Gr(r_1,\cE_{1,1})$. We have well defined multiplication maps
\begin{align*}
    &\cT_1\otimes H^0(\cO(d'_{2} - d'_{1})) \longrightarrow \pi_1^*\cE_{1,2} \longrightarrow \cE_{2,2} \\
    &\cT_1\otimes H^0(\cO(d'_{3} - d'_{1})) \longrightarrow \pi_1^*\cE_{1,3} \longrightarrow \cE_{2,3} \\
    &\vdots \\
    &\cT_1\otimes H^0(\cO(d'_{\ell} - d'_{1})) \longrightarrow \pi_1^*\cE_{1,\ell} \longrightarrow \cE_{2,\ell} 
\end{align*}
where the vector bundles $\cE_{2,i}$ on the right are by definition the cokernel of the multiplication map. Their fibers should be thought as vector spaces of forms of degree $d'_{i}$ up to multiples of certain forms of degree $d'_{1}$.

Let $\pi_2:\Gr(r_2,\cE_{2,2})\to\Gr(r_1,\cE_{1,1}) $ be the Grassmannian bundle of subbundles of rank $r_2$ in the vector bundle $\cE_{2,2}$, and let $\cT_2$ be the associated tautological bundle on $\Gr(r_2,\cE_{2,2})$. Then again we have well defined multiplication maps
\begin{align*}
    &\cT_2\otimes H^0(\cO(d'_{3} - d'_{2})) \longrightarrow \pi_2^*\cE_{2,3} \longrightarrow \cE_{3,3} \\
    &\cT_2\otimes H^0(\cO(d'_{4} - d'_{2})) \longrightarrow \pi_2^*\cE_{2,4} \longrightarrow \cE_{3,4} \\
    &\vdots \\
    &\cT_2\otimes H^0(\cO(d'_{\ell} - d'_{2})) \longrightarrow \pi_2^*\cE_{2,\ell} \longrightarrow \cE_{3,\ell}. 
\end{align*}
We can construct a Grassmannian bundle $\pi_3:\Gr(r_3,\cE_{3,3})\to\Gr(r_2,\cE_{2,2})$ and repeat the process. This eventually leads to a tower of Grassmannian bundles
\begin{equation}\label{eq:tower 2}
    \Gr(r_\ell,\cE_{\ell,\ell}) \xrightarrow{\pi_\ell} \cdots \xrightarrow{\pi_3} \Gr(r_2,\cE_{2,2}) \xrightarrow{\pi_2} \Gr(r_1,\cE_{1,1}) \xrightarrow{\pi_1} \spec{k}. 
\end{equation} 
Observe that the sheaves $\cE_{i,i}$ appearing in (\ref{eq:tower 2}) coincide with the sheaves $\cE_{i}$ that are in (\ref{eq:tower}) and the two towers of Grassmannian bundles are actually the same.

\subsection{A useful construction}
As before, pick $n\geq 2$ and let $\bfd=(d_1,d_2,\dots,d_r)$ be an $r$-uple of positive integers satisfying $d_1\leq d_2\leq\cdots\leq d_r$ and such that $0<r<n$. Then we define
\[ V(\bfd,n) := H^0(\PP^n,\cO(d_1))\times H^0(\PP^n,\cO(d_2))\times\cdots\times H^0(\PP^n,\cO(d_r)). \]
Observe that we can rewrite $V(\bfd, n)$ as
\[ H^0(\PP^n,\cO(d'_{1}))^{\times r_1} \times H^0(\PP^n,\cO(d'_{2}))^{\times r_2} \times \cdots \times H^0(\PP^n,\cO(d'_{\ell}))^{\times r_\ell} \]
so that $\GL_{r_i}$ acts by left multiplication on the $i^{\rm th}$-factor in the decomposition above. This defines an action of the group $\GL_{\bfd}:=\prod_{i=1}^{\ell} \GL_{r_i}$ on $V(\bfd, n)$.

Let $U(\bfd, n)\subset V(\bfd, n)$ be the open subscheme of polynomials $(f_1,\dots, f_r)$ such that the $f_1,\dots,f_{r_1}$ are linearly independent in $H^0(\PP^n,\cO(d'_{1}))$, the $f_{r_1+1},\dots,f_{r_1+f_2}$ are linearly independent in $H^0(\PP^n,\cO(d'_{2}))$, and so on. Then $\GL_{\bfd}$ acts freely on $U(\bfd, n)$, and we have
\[ [U(\bfd,n)/\GL_{\bfd}]\simeq \Gr(r_1,H^0(\cO(d'_{1})))\times\Gr(r_2,H^0(\cO(d'_{2})))\times\cdots\times\Gr(r_\ell,H^0(\cO(d'_{\ell}))). \]
We denote this last space as $\Gr(\bfd,n)$.

Consider the trivial Grassmannian bundle $$\Gr(r_2,H^0(\cO(d'_{1})))\times\Gr(r_1,H^0(\cO(d'_{2})))\to \Gr(r_1,H^0(\cO(d'_{2}))).$$ Let $U_2\subset \Gr(r_2,H^0(\cO(d'_{1})))\times\Gr(r_1,H^0(\cO(d'_{2})))$ be the open subscheme consisting of pairs $([E],[F])$ such that $E\cap (F \cdot H^0(\cO(d'_{2}-d'_{1})))=\{0\}$. Thinking of $\Gr(r_2,\cE_2)$ as in \ref{subsubsec:another}, we see that there is a well defined map
\[ q_1:U_2 \longrightarrow \Gr(r_2,\cE_{2}),\quad ([E],[F])\longmapsto ([\overline{E}],[F]) \]
where $\overline{E}$ is the image of $E$ along the quotient map $H^0(\cO(d'_{2}))\to \cE_{2,2}=\cE_2$. It follows from \Cref{prop:affine} that $U_2$ is an affine bundle over $\Gr(r_2,\cE_2)$.

The pullback of $U_2$ along $\pi_3:\Gr(r_3,\cE_{3})\to \Gr(r_2,\cE_{2})$ is an affine bundle over $\Gr(r_3,\cE_{3})$. Consider the product $\Gr(r_3,H^0(\cO(d'_{3})\times \pi_3^*U_2$ and let $U_3$ be the open subscheme in this product consisting of triples $([E_3],[E_2],[E_1])$ such that the intersection of $E_3$ with the vector subspace $H^0(\cO(d'_{3}-d'_{2}))\cdot E_2 + H^0(\cO(d'_{3}-d'_{1}))\cdot E_1$ has dimension zero. In particular, we have that $U_3$ is an open subscheme of the product $\Gr(r_3,H^0(\cO(d'_{3}))\times \Gr(r_2,H^0(\cO(d'_{2}))\times \Gr(r_3,H^0(\cO(d'_{3}))$.

We have a well defined map
\[ q_3: U_3 \longrightarrow \pi_3^*U_2,\quad ([E_3],[E_2],[E_1])\mapsto ([\overline{E_3}],[E_2],[E_1]) \]
where $\overline{E_3}$ is the image of $E_3$ in the vector space obtained by quotiening $H^0(\cO(d'_{3}))$ by the aforementioned vector subspace $H^0(\cO(d_3-d_2))\cdot E_2 + H^0(\cO(d'_{3}-d'_{1}))\cdot E_1$. Again by \Cref{prop:affine} this makes $U_3$ into an affine bundle over $\pi_3^*U_2$, hence it is also an affine bundle over $\Gr(r_3,\cE_{3})$. Repeating this process, we deduce the following.
\begin{proposition}\label{prop:U affine over Gr}
There exists an open subscheme $U_{\ell}$ of \[\Gr(r_1,H^0(\cO(d'_{1})))\times\Gr(r_2,H^0(\cO(d'_{2})))\times\cdots\times\Gr(r_\ell,H^0(\cO(d'_{\ell})))\]
which is an affine bundle over $\Gr(r_\ell,\cE_{\ell})$. Moreover, for $G=\GL_{n+1}$, $\SL_{n+1}$ or $\PGL_{n+1}$, this affine bundle is equivariant with respect to the $G$-action on $U_{\ell}$ and the $G$-action on the target.
\end{proposition}
In the Proposition above, the $G$-action on $U_\ell$ is induced by the $G$-action on the product of Grassmannians,

Summarizing, we have the following fundamental commutative diagram of $G$-schemes, when $G=\GL_{n+1}$ or $\SL_{n+1}$:
\[
\begin{tikzcd}
& U(\bfd,n)\ar[r, hook, "\text{open}"] \ar[d, "\GL_{\bfd}\text{-torsor}"] & V(\bfd,n) \\
U_{\ell} \ar[r, hook, "\text{open}"] \ar[d, "\text{affine bundle}"] & \Gr(\bfd,n) & \\
\Gr(r_{\ell},\cE_{\ell}) & & 
\end{tikzcd}
\]
This will be helpful for computing equivariant Chow rings in  the next Sections.

\subsection{Discriminant divisors}\label{sub:disc}

Let $H_{d_i}\subset H^0(\PP^n,\cO(d_i))\times\PP^n$ be the universal hypersurface of degree $d_i$ and let ${\rm pr}_i:V(\bfd,n)\times \PP^n\to H^0(\PP^n,\cO(d_i))\times\PP^n $ be the projection morphism. We define a subscheme $X_{\bfd,n}\subset V(\bfd,n)\times\PP^n$ as the (schematic) intersection ${\rm pr}_1^{-1}(H_{d_1})\cap \cdots \cap {\rm pr}_r^{-1}(H_{d_r})$. The fiber of $X_{\bfd,n}\to V(\bfd,n)$ over a point $(f_1,\dots,f_r)$ is the projective scheme defined by the homogeneous ideal $I=(f_1,\dots,f_r)$.

We will denote $Z_{\bfd,n}$ the (schematic) singular locus of the morphism $X\to V(\bfd,n)$: in particular, the points of $Z_{\bfd,n}$ are tuples $(f_1,\dots,f_r)$ such that the projective scheme in $\PP^n$ defined by the homogeneous ideal $I=(f_1,\dots,f_r)$ is either singular or of codimension $>r$. This divisor is invariant with respect to the $\GL_{n+1}$-action on $V(\bfd,n)$.

Observe that the action of $\GL_{\bfd}$ is free on $X_{\bfd,n}\cap (U(\bfd,n)\times \PP^n)$, hence we have a well defined quotient scheme $Y_{\bfd,n}=X_{\bfd,n}\cap (U(\bfd,n)\times \PP^n)/\GL_{\bfd}$, which can be regarded as a subscheme of $\Gr(\bfd,n)\times\PP^n$.

The divisor $Z_{\bfd,n}$ is invariant with respect to the $\GL_{\bfd}$-action on $V(\bfd,n)$: we denote by $D_{\bfd, n}$ the divisor in $\Gr(\bfd,n)$ defined as the geometric quotient $Z_{\bfd,n}\cap U(\bfd, n)/\GL_{\bfd}$. The points of $D_{\bfd,n}$ are tuples of vector subspaces $$(\langle f_1,\dots,f_{r_1}\rangle,\dots,\langle f_{r-r_{\ell}+1},\dots,f_{r}\rangle )$$ such that the projective scheme in $\PP^n$ defined by the homogeneous ideal $I=(f_1,\dots,f_r)$ is either singular or of codimension $>r$. Again, the divisor $D_{\bfd, n}$ is invariant with respect to the $\GL_{n+1}$-action on the product of Grassmannians $\Gr(\bfd,n)$.

The open subscheme $U_{\ell}\subset\Gr(\bfd,n)$ is an affine bundle over $\Gr(r_{\ell},\cE_{\ell})$ and $Y_{\bfd,n}\cap (U_{\ell}\times\PP^n)$ descends along this affine bundle: in this way we obtain again the subscheme $\cY_{\bfd,n}\subset \Gr(r_{\ell},\cE_{\ell})\times\PP^n$. In particular, the preimage of $S_{\bfd,n}$ in $U_{\ell}$ is equal to $D_{\bfd,n}\cap U_{\ell}$.
Putting all together, we get the following.
\begin{lemma}\label{lemma:diagram}
Let $G$ be either $\GL_{n+1}$ or $\SL_{n+1}$. Then the following diagram of $G$-schemes holds:
\[
\begin{tikzcd}
& U(\bfd,n)\smallsetminus (Z_{\bfd,n}\cap U(\bfd,n)) \ar[r, equal] \ar[d, "\GL_{\bfd}\text{-torsor}"] & V(\bfd,n)\smallsetminus Z_{\bfd,n} \\
U_{\ell}\smallsetminus (D_{\bfd,n}\cap U_{\ell}) \ar[r, equal] \ar[d, "\text{affine bundle}"] & \Gr(\bfd,n)\smallsetminus D_{\bfd,n} & \\
\Gr(r_{\ell},\cE_{\ell})\smallsetminus S_{\bfd,n} & & 
\end{tikzcd}
\]
\end{lemma}
Again, we will need this for intersection-theoretical computations.

\section{Chow rings of moduli of smooth complete intersections}\label{sec:chow}
In this Section we give a presentation of the Chow ring of $\cM^G_n(\bfd)$ in terms of generators and relations (\Cref{thm:chow}), for $G=\GL_{n+1}$, $\SL_{n+1}$ or $\PGL_{n+1}$. The relations are not explicit, in the sense that we do not express them via closed formulas involving the generators and the quantities $\bfd$ and $n$; nevertheless, once these values are fixed, the relations can be practically computed. 

We give two quick examples of concrete computations in \Cref{prop:chow M5} and \Cref{prop:chow U8}: in the first Proposition we reprove the Theorem of Izadi the Chow ring of $M_5$, the moduli space of genus five curves, using \Cref{thm:chow}; in the second one, we study the Chow ring of $K_8$, the moduli space of polarized K3 surfaces of degree eight.

For a quick recap of equivariant Chow groups and their properties, the reader can consult \cite{DL-coh}*{\S 5.1}.
\subsection{A resolution of $Z_{\bfd,n}$}\label{sec:resolution}
Recall that we defined
\[ V(\bfd,n) := H^0(\PP^n,\cO(d_1))\times H^0(\PP^n,\cO(d_2))\times\cdots\times H^0(\PP^n,\cO(d_r)). \]
and that $Z_{\bfd,n}\subset V(\bfd,n)$ is the singular locus of the map $X_{\bfd,n}\to V(\bfd,n)$ whose fiber over a point $(f_1,\dots,f_r)$ is the complete intersection $\{f_1=\cdots=f_r=0\}$ in $\PP^n$ (see \ref{sub:disc}).
In \ref{sec:other relations} we will need to compute the generators of the image of the pushforward of equivariant Chow groups
\[ {\rm CH}^{*-1}_{\GL_{n+1}\times \GL_{\bfd}}(Z_{\bfd, n})\longrightarrow \ch_{\GL_{n+1}\times\GL_{\bfd}}(V(\bfd,n)).  \]
For this reason, we are going to construct an equivariant resolution $\widetilde{Z}_{\bfd,n}\to Z_{\bfd,n}$ with the property that the Chow ring of the domain admits a particularly nice presentation.

Set $s=n-r+2$ and let ${\rm{Mat}}_{n+1,s}$ be the vector space of matrices with $n+1$ rows and $s$ columns and define ${\rm{P}}_s\subset \GL_s$ as the parabolic subgroup formed by those matrices whose first column is zero except for the first entry. The group ${\rm{P}}_s$ acts linearly on ${\rm{Mat}}_{n+1,s}$ via the formula $B\cdot Q:=QB^{-1}$. 

Let ${\rm{Mat}}_{n+1,s}^o \subset \ {\rm{Mat}_{n+1,s}}$ be the open subscheme formed by matrices of maximal rank: then ${\rm P}_s$ acts freely on ${\rm{Mat}}_{n+1,s}^o$ and the quotient is isomorphic to the flag variety ${\rm Fl}_{n+1,s}$ parametrizing partial flags $L\subset F$ in $\mathbb{A}^{n+1}$ where $\dim(L)=1$ and $\dim(F)=n-r+2$.

Observe also that the group $\GL_{n+1}$ acts on ${\rm{Mat}}_{n+1,s}$ via left multiplication. This action descends to the an action on the product $\Gr(\bfd,n)\times {\rm Fl}_{n+1,s}$.

A point in $V(\bfd, n)\times {\rm{Mat}}_{n+1,s}$ amounts to $r$ homogeneous forms $f_1,\dots,f_r$ with $\deg(f_i)=r_i$ together with an $(n+1)\times s$ matrix
\[ Q=
\begin{pNiceArray}{c|c|c|c}
  &  & &  \\
 q_1 & q_2 & \hdots & q_r\\
  &  & &
\end{pNiceArray}
=
\begin{pNiceMatrix}
  q_{1,0} & q_{2,0}  & \hdots & q_{r,0} \\
 q_{1,1} & q_{2,1} & \hdots & q_{r,1} \\
  \vdots &  \vdots  &  &\vdots \\
  q_{1,n} & q_{2,n} & \hdots & q_{r,n}
\end{pNiceMatrix}.
\] 
Let $\widetilde{W}_{\bfd,n}\subset V(\bfd, n)\times{\rm{Mat}}_{n+1,s}$ be the closed subscheme of tuples $(f_1,\dots,f_r,Q)$ satisfying the matrix equation $J(f_1,\dots,f_r)(q_1)\cdot Q=0$, that is
\[
\begin{pNiceMatrix}
  \partial_{x_0}f_1(q_1) & \partial_{x_2}f_1(q_1)  & \hdots & \partial_{x_n}f_1(q_1) \\
  \partial_{x_0}f_2(q_1) & \partial_{x_2}f_2(q_1)  & \hdots & \partial_{x_n}f_2(q_1) \\
  \vdots &  \vdots  &  &\vdots \\
 \partial_{x_0}f_r(q_1) & \partial_{x_2}f_r(q_1)  & \hdots & \partial_{x_n}f_r(q_1)
\end{pNiceMatrix}
\begin{pNiceMatrix}
  q_{1,0} & q_{2,0}  & \hdots & q_{r,0} \\
 q_{1,1} & q_{2,1} & \hdots & q_{r,1} \\
  \vdots &  \vdots  &  &\vdots \\
  q_{1,n} & q_{2,n} & \hdots & q_{r,n}
\end{pNiceMatrix}= 0\]
where $J(f_1,\dots,f_r)(q_1)$ is the Jacobian matrix associated to the form $f_1,\dots,f_r$ evaluated at the vector $q_1$. We can interpret $\widetilde{W}_{\bfd,n}$ as the closed subscheme of tuples $(f_1,\dots,f_r,Q)$ such that the point $[q_1]$ in $\PP^n$ is a singular point for the projective scheme defined by the homogeneous ideal $I=(f_1,\dots,f_r)$, and the vector subspace $Q\subset \mathbb{A}^{n+1}$ contains $q_1$ and is contained in the kernel of the Jacobian matrix. This subscheme is $\GL_{n+1}$-invariant.

The geometric quotient of $\widetilde{W}_{\bfd,n}\cap (V(\bfd,n)\times{\rm Mat}_{n+1,s}^o)$  by the (free) ${\rm P}_s$-action is a closed subscheme of $V(\bfd,n)\times {\rm Fl}_{n+1,s}$ which we denote $\widetilde{Z}_{\bfd,n}$. By construction, the points of $\widetilde{Z}_{\bfd ,n}$ correspond to tuples $((f_1,\dots,f_r),p\in E\subset \PP^n)$ such that $p$ is a singular point of the projective subscheme $\{f_1=f_2=\cdots=f_r=0\}$ and $E\simeq \PP^{s-1}$ is a projective subspace contained in the projective variety defined by the matrix equation $J(f_1,\dots,f_r)(p)\cdot (x_0,\dots ,x_n)=0$ (although $J(f_1,\dots,f_r)(p)$ is not well defined, the projective variety is actually well defined). Observe that the projection on $V(\bfd,n)$ induces a map $\widetilde{Z}_{\bfd,n}\to Z_{\bfd,n}$.
\begin{lemma}\label{lemma:resolution}
The pushforward morphism
\[ \ch_{\GL_{n+1}\times\GL_{\bfd}}(\widetilde{Z}_{\bfd,n}) \longrightarrow \ch_{\GL_{n+1}\times\GL_{\bfd}}(Z_{\bfd,n}) \]
is surjective. Moreover the pushforward of the fundamental class $[\widetilde{Z}_{\bfd,n}]$ along $\pr_1:V(\bfd,n)\times {\rm Fl}_{n+1,s}\to V(\bfd,n)$ is equal to $[Z_{\bfd,n}]$.
\end{lemma}
\begin{proof}
This follows from the fact that $\widetilde{Z}_{\bfd,n}\to Z_{\bfd,n}$ is surjective and birational.
\end{proof}

\subsection{Relations}\label{sec:other relations}
From \Cref{lemma:diagram} we know that $\Gr(\bfd,n)\smallsetminus D_{\bfd,n}\to \Gr(r_{\ell},\cE_{\ell})\smallsetminus S_{\bfd, n}$ is a $\GL_{n+1}$-equivariant affine bundle. This implies that
\begin{align*}
    \ch_{\GL_{n+1}}(\Gr(r_{\ell},\cE_{\ell})\smallsetminus S_{\bfd, n})\simeq \ch_{\GL_{n+1}}(\Gr(\bfd,n)\smallsetminus D_{\bfd,n}),
\end{align*}
because the pullback along an affine bundle induces an isomorphism of Chow rings.
\Cref{lemma:diagram} also implies that we have an isomorphism
\[ \ch_{\GL_{n+1}\times\GL_{\bfd}}(V(\bfd,n)\smallsetminus Z_{\bfd,n}) \simeq \ch_{\GL_{n+1}}(\Gr(\bfd,n)\smallsetminus D_{\bfd,n})  \]
because $V(\bfd,n)\smallsetminus Z_{\bfd,n}\to \Gr(\bfd,n)\smallsetminus D_{\bfd,n}$ is a $\GL_\bfd$-torsor.
Putting all together, we deduce
\[ \ch_{\GL_{n+1}}(\Gr(r_\ell,\cE_\ell)\smallsetminus S_{\bfd,n}) \simeq \ch_{\GL_{n+1}\times\GL_{\bfd}}(V(\bfd,n)\smallsetminus Z_{\bfd,n}). \]

The localization sequence for equivariant Chow groups tells us that we have a short exact sequence of groups
\[  {\rm CH}^{*-1}_{\GL_{n+1}}(Z_{\bfd,n})\longrightarrow  \ch_{\GL_{n+1}\times\GL_{\bfd}}(V(\bfd,n)) \longrightarrow \ch_{\GL_{n+1}\times\GL_{\bfd}}(V(\bfd,n)\smallsetminus Z_{\bfd,n})\to 0, \]
so that the ideal of relations $R$ appearing in the formula
\begin{equation}\label{eq:R}
    \ch_{\GL_{n+1}\times\GL_{\bfd}}(V(\bfd,n)\smallsetminus Z_{\bfd,n})\simeq \ch_{\GL_{n+1}\times\GL_{\bfd}}(V(\bfd,n))/R,
\end{equation}  
is generated by the puhforward of cycles in $Z_{\bfd,n}$.
\Cref{lemma:resolution} implies that $R$ is equal to the image of the pushforward
\[  {\rm CH}^{*-1}_{\GL_{n+1}\times\GL_{\bfd}}(\widetilde{Z}_{\bfd,n}) \longrightarrow  \ch_{\GL_{n+1}\times\GL_{\bfd}}(V(\bfd,n)). \]
Observe that $\widetilde{Z}_{\bfd,n}$ is an equivariant vector bundle over the flag variety ${\rm Fl}_{n+1,s}$. This implies that the generators of the equivariant Chow ring of $\widetilde{Z}_{\bfd,n}$ as $\ch_{\GL_{n+1}\times\GL_{\bfd}}$-module are obtained by pulling back the generators of $\ch_{\GL_{n+1}\times\GL_{\bfd}}({\rm Fl}_{n+1,s})$ as $\ch_{\GL_{n+1}\times\GL_{\bfd}}$-module.

Recall that the flag variety ${\rm Fl}_{n+1,s}$ has a universal partial flag $\cL\subset\cF\subset\cO^{\oplus n+1}$ of equivariant locally free sheaves, where $\cL$ has rank $1$ and $\cF$ has rank $s$. The equivariant Chow ring of the flag variety is generated, as $\ch_{\GL_{n+1}\times\GL_{\bfd}}$-module, by monomials in the Chern classes of $\cL$ and $\cF/\cL$.

More precisely, set $\beta_1=c_1(\cL)$ and $b_i=c_i(\cF/\cL)$ for $i=1,\dots,s-1$: then the generators are given by polynomials $P(\beta_1,b_1,b_2,\dots,b_{s-1})$ and thanks to the relations in the flag variety, we can restrict ourselves to monomials that have degree $< s$ in $\beta_1$ and total degree $\leq rs-1$.

Consider the equivariant diagram
\[
\begin{tikzcd}
\widetilde{Z}_{\bfd,n} \ar[r, hook, "i"] \ar[dr] & V(\bfd,n)\times {\rm Fl}_{n+1,s} \ar[d, "{\rm pr}_1"] \ar[r, "{\rm pr}_2"] & {\rm Fl}_{n+1,s} \ar[d, "\pi"] \\
& V(\bfd,n) \ar[r] & \spec{k}
\end{tikzcd}
\]
Then the ideal of relations $R$ appearing in (\ref{eq:R}) is generated by expressions of the form 
\[{\rm pr}_{1*}i_*(i^*{\rm pr}_2^*P(\beta_1,b_1,\dots,b_{s-1})),\]
where $P$ is a monomial as described before. The projection formula readily implies that these expressions can be rewritten as 
\[ {\rm pr}_{1*}([\widetilde{Z}_{\bfd,n}]\cdot {\rm pr}_2^*P(\beta_1,b_1,\dots,b_{s-1})),\]
where $[\widetilde{Z}_{\bfd,n}]$ is the $\GL_{n+1}\times\GL_{\bfd}$-equivariant fundamental class of $\widetilde{Z}_{\bfd,n}$.

\subsubsection{The fundamental class of $\widetilde{Z}_{\bfd,n}$}\label{sub:Zdn}
Consider the cartesian diagram of quotient stacks
\[ \xymatrix{
[\Gr(\bfd, n)\times {\rm Fl}_{n+1,s}/\GL_{n+1}] \ar[r]^{\hspace{25pt} \pr_2} \ar[d]^{\pr_1} & [{\rm Fl}_{n+1,s}/\GL_{n+1}] \ar[d]^{\pi} \\
[\Gr(\bfd, n)/\GL_{n+1}] \ar[r] & \cB\GL_{n+1}
}\]
Our goal is to compute $[\widetilde{Z}_{\bfd,n}]$, the $\GL_{n+1}$-equivariant fundamental class of $\widetilde{Z}_{\bfd,n}$.

In \ref{sec:resolution} we introduced the parabolic subgroup ${{\rm P}}_s\subset \GL_{s}$. The quotient of the open subset ${{\rm Mat}}_{n+1,s}^\circ\subset {{\rm Mat}}_{n+1,s} $ of matrices of maximal rank by the natural action of ${{\rm P}}_s$ is the flag variety ${{\rm Fl}}_{n+1,s}$.

Let $\Gamma_{\bfd,n}:=\GL_{n+1}\times \GL_{\bfd} \times {\rm P}_s$. Then we have
\[ \ch_{\GL_{n+1}\times\GL_\bfd}(V(\bfd,n)\times {\rm Fl}_{n+1,s})\simeq \ch_{\Gamma_{\bfd,n}}(V(\bfd,n)\times {\rm{Mat}}^o_{n+1,s}). \]
Moreover, there is a surjection of equivariant Chow rings
\begin{equation}\label{eq:surjection} \ch_{\Gamma_{\bfd,n}}(V(\bfd, n)\times{\rm{Mat}}_{n+1,s})\longrightarrow \ch_{\Gamma_{\bfd, n}}(U(\bfd,n)\times {\rm{Mat}}_{n+1,s}^o) \end{equation}
given by the pullback along the open embedding. 
As $[V(\bfd, n)\times{\rm {Mat}}_{n+1,s}/\Gamma_{\bfd,n}]$ is a vector bundle over $\cB \Gamma_{\bfd,n}$, we deduce that
\[ \ch_{\Gamma_{\bfd,n}}(V(\bfd, n)\times{\rm{Mat}}_{n+1,s})\simeq \ch(\cB\Gamma_{\bfd,n}). \]
Recall (\cite{EG}*{\S3.2}) that the Chow ring of $\cB \GL_m$ is isomorphic to the ring of polynomials in the Chern classes of the universal rank $m$ vector bundle $\cV_m \to \cB\GL_m$.

The stack $\cB {\rm P}_s$ classifies vector bundles of rank $s$ together with a subbundle of rank $1$. Let $\cF$ be the universal vector bundle of rank $s$ on $\cB {\rm P}_s$, let $\cL\subset\cF$ be the universal vector subbundle  of rank $1$, so that $\cF/\cL$ is a universal quotient bundle of rank $s-1$. 

Then it easily follows from \cite{EG}*{Proposition 6} that the Chow ring of $\cB {\rm P}_s$ is the ring of polynomials in the first Chern class of $\cL$ and in the Chern classes of $\cF/\cL$. As $\ch(\cB(G\times H))\simeq \ch(\cB G)\otimes\ch(\cB H)$, we deduce
\begin{equation}\label{eq:chow VxF} \ch(\cB\Gamma_{\bfd,n})\simeq \QQ[\{c_i\}_{i\leq n+1},\{a_{j,k}\}_{j\leq \ell,k\leq r_j},\beta_1,\{ b_m \}_{m\leq s-1} ]. \end{equation}

In \ref{sec:resolution} we defined the closed subscheme $\widetilde{W}_{\bfd,n}\subset V(\bfd,n)\times{\rm Mat}_{n+1,s}$.
Observe that the $\Gamma_{\bfd,n}$-equivariant class of $\widetilde{W}_{\bfd,n}$ is sent to the $\GL_{n+1}$-equivariant fundamental class of $\widetilde{Z}_{\bfd,n}$ by the surjection (\ref{eq:surjection}), hence if we compute $[\widetilde{W}_{\bfd,n}]_{\Gamma_{\bfd,n}}$ in terms of the generators appearing in (\ref{eq:chow VxF}), we will also get an explicit expression for $[\widetilde{Z}_{\bfd,n}]$.

\subsubsection{The fundamental class of $\widetilde{W}_{\bfd,n}$}
We know from \cite{EG}*{Proposition 6} that for every special group $G$ with maximal subtorus $T$ and every smooth scheme $X$ endowed with a $G$-action, there is an inclusion of rings $\ch_G(X)\hookrightarrow\ch_T(X)$ whose image corresponds to the subring of $W$-invariant element, where $W$ is the Weyl group associated to $T\subset G$. 

In particular, if $Y\subset X$ is a $G$-invariant variety, the image of $[Y]_G$ in $\ch_T(X)$ is equal to $[Y]_T$; in other terms, by knowing an explicit expression for $[Y]_T$, one immediately get a formula for $[Y]_G$ by just rewriting that expression in terms of the $W$-invariant generators. We apply this argument to compute $[\widetilde{W}_{\bfd,n}]_{\Gamma_{\bfd, n}}$.

Let $T_{n,r}\subset \Gamma_{\bfd, n}$ be the maximal subtorus of diagonal matrices. We have $T_{n,r}\simeq \gm^{n+1}\times \gm^r \times \gm^s$, and
\begin{equation} \label{eq:chow BT} \ch(\cB T_{n,r}) \simeq \ch_{T_{n,r}}(\spec{k}) \simeq \QQ[t_1,\dots,t_{n+1},\gamma_1,\dots,\gamma_r,\beta_1,\dots,\beta_s]. \end{equation}
where the $t_i$ are pulled back from $\ch(\cB \gm^{n+1})$, the $\gamma_j$ from $\ch(\cB \gm^r)$ and the $\beta_k$ from $\ch(\cB\gm^s)$. For the latters, we adopt the convention that $\beta_k$ is the first Chern class of the $\gm$-representation of weight $-1$, as this choice is slightly more convenient for future computations.

\begin{remark}
The generators appearing in (\ref{eq:chow VxF}) can be rewritten in terms of the generators appearing in (\ref{eq:chow BT}) as follows: the elements $c_i$ are the elementary symmetric polynomials of degree $i$ in $t_1,\dots,t_{n+1}$; the elements $\alpha_{j,k}$ are the elementary symmetric polynomials of degree $k$ in $\gamma_{r_1+\cdots r_{j-1}+1},\dots,\gamma_{r_1+\cdots+r_j}$; the elements $b_m$ are the elementary symmetric polynomial of degree $m$ in the $\beta_2,\dots,\beta_s$ multiplied by $(-1)^m$, and the two $\beta_1$ coincide. 
\end{remark}
Computing $[\widetilde{W}_{\bfd,n}]_{T_{n,r}}$ is quite easy, because $\widetilde{W}_{\bfd,n}$ is a complete intersection of $T_{n,r}$-invariant hypersurfaces $H_{i,j}$ of equation $$F_{i,j} = \sum_{k=0}^{n} \partial_{x_k} f_i(q_0)\cdot q_{j,k} =0.$$
Observe that an element $(\lambda_1,\dots,\lambda_{n+1},\mu_1,\dots,\mu_r,\nu_1,\dots,\nu_s)$ of $T_{n,r}$ acts on a polynomial $F_{i,j}$ as
\[ F_{i,j} \longmapsto \mu_i\beta_0^{1-d_i}\beta_j^{-1} F_{i,j}, \]
hence by \cite{DLFV}*{Lemma 2.6} the fundamental class of $H_{i,j}$ is $\gamma_i + (d_i-1)\beta_1 + \beta_j$. We deduce
\begin{equation}\label{eq:Wtilde}
    [\widetilde{W}_{\bfd,n}]_{T_{n,r}} = \prod_{1\leq i \leq r, 1\leq j\leq s} [H_{i,j}]_{T_{n,r}} = \prod_{1\leq i\leq r, 1\leq j \leq s} (\gamma_i + (d_i-1)\beta_1 + \beta_j).
\end{equation}
\begin{remark}
Observe that this formula has the symmetries we expected it to have, i.e. is invariant with respect to the Weyl group associated to the torus $T_{n,r}\subset\Gamma_{\bfd,n}$. 

In particular, an explicit expression of $[\widetilde{Z}]_{\GL_{n+1}}$ can be obtained by rewriting (\ref{eq:Wtilde}) using the generators appearing in (\ref{eq:chow VxF}). On the other hand, the formulation in (\ref{eq:Wtilde}) is quite more manageable from a computational point of view.
\end{remark}

Expanding the expression in (\ref{eq:Wtilde}) we get the following.
\begin{lemma}\label{lm:W}
We have
   \[ [\widetilde{Z}_{\bfd,n}]=[\widetilde{W}_{\bfd,n}]_{T_{n,r}} = \sum_{0\leq a_1,\dots,a_r \leq s} \gamma_1^{a_1}\cdots\gamma_r^{a_r}\cdot C_s(a_1,\dots,a_r) \]
where
\[
    C_s(a_1,\dots,a_r)=\prod_{i=1}^{r} \sigma_{s-a_i}((d_i-1)\beta_1+\beta_1,\dots,(d_i-1)\beta_1+\beta_s).
\]
and the $\sigma_d$ stand for the elementary symmetric polynomials of degree $d$.
\end{lemma}

\begin{remark}
The formula for the equivariant fundamental class in \Cref{lm:W} is a priori a formula in the $\beta_i$. Nevertheless, it is symmetric in these variables, hence it is actually a polynomial in $\beta_1$ and $b_i=(-1)^i\sigma_i(\beta_2,\dots,\beta_s)$ for $i=1,\dots,s-1$.
\end{remark}

\subsubsection{End of the computation}
Recall that we have an equivariant diagram
\[
\begin{tikzcd}
\widetilde{Z}_{\bfd,n} \ar[r, hook, "i"] \ar[dr] & V(\bfd,n)\times {\rm Fl}_{n+1,s} \ar[d, "{\rm pr}_1"] \ar[r, "{\rm pr}_2"] & {\rm Fl}_{n+1,s} \ar[d, "\pi"] \\
& V(\bfd,n) \ar[r] & \spec{k}.
\end{tikzcd}
\]
and that the Chow ring of the flag variety ${\rm Fl}_{n+1,s}$ is algebraically generated over $\ch(\cB(\GL_{n+1}\times\GL_\bfd))$ by $\beta_1$, the first Chern class of the tautological line bundle, and by $b_1,\dots,b_{s-1}$, the Chern classes of the tautological quotient bundle. Recall also that $b_i=(-1)^i\sigma_i(\beta_2,\dots,\beta_s)$, i.e. the class $b_i$ is up to a sign the symmetric polynomial in its Chern roots $\beta_2,\dots,\beta_s$.

Recall also from \ref{sub:Zdn} that $\ch(\cB(\GL_{n+1}\times\GL_\bfd))$ is the ring of polynomials in the variables $c_1,\dots,c_{n+1}$ (the generators that come from $\ch(\cB\GL_{n+1})$) and in the $a_{j,k}$ for $j\leq \ell$ and $k\leq r_j$: the latters are the elementary symmetric polynomials of degree $k$ in $\gamma_{r_1+\cdots r_{j-1}+1},\dots,\gamma_{r_1+\cdots+r_j}$. In other words, if we let $\mathfrak{S}_{\bfd}:=\prod_{i=1}^{\ell} \mathfrak{S}_{r_i}$ be a product of symmetric groups, we can write
\[ \ch(\cB\GL_{\bfd}) \simeq \QQ[\{ a_{j,k} \} ] = \QQ[\gamma_1,\dots,\gamma_r]^{\mathfrak{S}_\bfd}  \]
where $\mathfrak{S}_{r_i}$ acts on $\gamma_{r_1+\cdots +r_{i-1}+1},\dots,\gamma_{r_1+\cdots +r_i}$ by permutation.

With this setup in mind, we are ready to state the main result of the Section.
\begin{theorem}\label{thm:chow}
We have
\[ \ch(\cM^{\GL}_n(\bfd))\simeq \QQ[c_1,\dots,c_{n+1},\gamma_1,\dots,\gamma_r]^{\mathfrak{S}_{\bfd}}/R   \]
where the ideal of relations $R$ is generated by cycles of the form
\[ \sum_{0\leq a_1,\dots,a_r \leq s} \gamma_1^{a_1}\cdots\gamma_r^{a_r}\cdot \pi_*\left( C_s(a_1,\dots,a_r)P(\beta_1,b_1,\dots,b_{s-1})\right) \]
for $P(\beta_1,b_1,\dots,b_{s-1})$ any monomial of degree $< s$ in $\beta_1$ and total degree $\leq rs -1$, and
\[ C_s(a_1,\dots,a_r)=\prod_{i=1}^{r} \sigma_{s-a_i}((d_i-1)\beta_1+\beta_1,\dots,(d_i-1)\beta_1+\beta_s) \]
\end{theorem}
\begin{proof}
\Cref{prop:presentation} for $G=\GL_{n+1}$ tells us that $$\cM^{\GL}_n(\bfd)\simeq [(\Gr(r_\ell,\cE_\ell)\smallsetminus S_{\bfd,n})/\GL_{n+1}],$$ hence the Chow ring of the stack on the left is isomorphic to  $\ch_{\GL_{n+1}}(\Gr(r_\ell,\cE_\ell)\smallsetminus S_{\bfd,n})$.
In \ref{sec:other relations} we have seen that
\[ \ch_{\GL_{n+1}}(\Gr(r_\ell,\cE_\ell)\smallsetminus S_{\bfd,n}) \simeq \ch_{\GL_{n+1}\times\GL_{\bfd}}(V(\bfd,n)\smallsetminus Z_{\bfd,n}), \]
and that the term on the right is isomorphic to
\begin{align*}
     \ch_{\GL_{n+1}\times\GL_\bfd}(V(\bfd,n))/R &\simeq \QQ[c_1,\dots,c_{n+1},\{a_{j,k}\}_{j\leq \ell,k\leq r_j}]/R \\
     &\simeq  \QQ[c_1,\dots,c_{n+1},\gamma_1,\dots,\gamma_r]^{\mathfrak{S}_{\bfd}}/R
\end{align*}
where $R$ is the ideal generated by cycles of the form
\[ {\rm pr}_{1*}([\widetilde{Z}_{\bfd,n}]\cdot {\rm pr}_2^*P(\beta_1,b_1,\dots,b_{s-1})). \]
In the formula above $P(\beta_1,b_1,\dots,b_{s-1})$ is any monomial of degree $< s$ in $\beta_1$ and total degree $\leq rs-1$.

\Cref{lm:W} gives us an explicit expression for $[\widetilde{Z}_{\bfd,n}]$. Applying the compatibility formula, we deduce that the generators of $R$ are
\[ \sum_{0\leq a_1,\dots,a_r \leq s} \gamma_1^{a_1}\cdots\gamma_r^{a_r}\cdot \pi_*\left( C_s(a_1,\dots,a_r)P(\beta_1,b_1,\dots,b_{s-1})\right). \]
This concludes the proof.
\end{proof}
\begin{remark}
The relations appearing in \Cref{thm:chow} actually holds in the \emph{integral} Chow ring, but we don't know whether they still generate the integral Chow ring, although this is probably not the case.
\end{remark}
\begin{corollary}\label{cor:chow PGL}
With the same notation as \Cref{thm:chow}, for $G=\SL_{n+1}$ or $\PGL_{n+1}$, we have
\[ \ch(\cM^G_n(\bfd))\simeq \QQ[c_1,\dots,c_{n+1},\gamma_1,\dots,\gamma_r]^{\mathfrak{S}_{\bfd}}/(R,c_1) \]
\end{corollary}
\begin{proof}
For $G=\SL_{n+1}$, the same argument used to prove \Cref{thm:chow} applies, with the only difference that $\ch_{\SL_{n+1}}(\spec{k})$ is the ring of polynomials in $c_2,\dots,c_{n+1}$. For the case $G=\PGL_{n+1}$, we use \cite{FV}*{Lemma 5.4}, which tells us that for every $\PGL_{n+1}$-scheme $X$, the kernel of the natural pull-back map $\ch_{\PGL_{n+1}}(X)\to\ch_{\SL_{n+1}}(X)$ is torsion, hence zero when the Chow groups are taken with $\QQ$-coefficients.
\end{proof}
\begin{remark}
In contrast with what happens for $G=\GL_{n+1}$ and $\SL_{n+1}$, it is not true that all the relations appearing in \Cref{cor:chow PGL} for $G=\PGL_{n+1}$ hold true in the \emph{integral} Chow ring.
\end{remark}
\subsection{First applications}
At first glance, \Cref{thm:chow} may look too abstract to be useful when it comes to computing explicit descriptions of Chow rings of moduli spaces. In particular, it lacks closed formulas for the generators of the ideal of relations that only involve the values $d_1,\dots,d_r$ and $n$. Nonetheless, once those values are fixed, it is quite easy to compute the relations using localization formulas (see \cite{EG-loc}*{Theorem 2}, \cite{DLFV}*{Remark 2.4}).

Given a cycle $\xi$ in $\ch_T({\rm Fl}_{n+1,s})$, where $T\subset\GL_{n+1}$ is the subtorus of diagonal matrices, the localization formula tells us that the pushforward along $\pi_*:\ch_T({\rm Fl}_{n+1,s})\to\ch_T(\spec{k})$ is equal to
\[ \pi_*\xi = \sum_{p\in {\rm Fl}_{n+1,s}^T} \frac{i_p^*\xi}{c_{{\rm top}}^T(T({\rm Fl}_{n+1,s})_p)} \]
where $i_p:p\hookrightarrow {\rm Fl}_{n+1,s}$ is the inclusion of a $T$-fixed point in the flag variety.

Although the right hand side is a priori only a rational function, the theory behind localization formulas assures us that the term on the right is actually a polynomial belonging to $\ch_T(\spec{k})\simeq\QQ[t_1,\dots,t_{n+1}]$. 

The set of $T$-fixed points of the flag variety is in bijection with the set of pairs $(i_1,I)$ where $1\leq i_1\leq n+1$ and $I=\{i_2,\dots,i_s\}$ is a subset of $\{1,2,\dots,n+1\}\smallsetminus\{i_1\}$ of length $s-1$: to such a pair, we associate the $T$-fixed point in ${\rm Fl}_{n+1,s}$ given by the flag
\[ \left(V_{i_1}\subset V_I \subset \AA^{n+1}\right)=\langle e_{i_1} \rangle \subset \langle e_{i_1},e_{i_2},\dots,e_{i_s}\rangle \subset \AA^{n+1} \]
where $e_1,\dots,e_{n+1}$ is the standard basis of $\AA^{n+1}$ as a vector space.

We can regard the flag variety as the projectivization of the tautological bundle over $\Gr(s,n+1)$. In this way we see that the tangent space of the fixed point associated to $(i_1,I)$ is 
\[\hom(V_{i_1},V_I/V_{i_1})\oplus\hom(V_I,\AA^{n+1}/V_I)\simeq \langle e_1^{\vee}\otimes e_{i_2},\dots,e_1^{\vee}\otimes e_{i_s}\rangle \oplus\langle \dots, e_{i_j}^{\vee}\otimes e_{i_k},\dots \rangle \]
where $i_j\in I\cup \{i_1\}$ and $i_k$ belongs to the complement of $I\cup\{i_1 \}$ in $\{1,\dots,n+1\}$. As $T$ acts on $e_m$ via multiplication by the character $t_m$, we deduce that
\[ c_{{\rm top}}^T(T({\rm Fl}_{n+1,s})_{(i_1,I)})=\prod_{i_m\in I} (t_{i_m}-t_{i_1}) \prod_{\substack{i_j\in I\cup \{i_1\} \\ i_k \in (I\cup\{i_1\})^c}} (t_{i_k}-t_{i_j}). \]
Recall that the $\beta_j$ for $j=1,\dots,s$ are the Chern roots of the dual of the tautological bundle on $\Gr(s,n+1)$. Therefore, their restriction to the $T$-equivariant Chow ring of the fixed point associated to $(i_1,I)$ are exactly the Chern roots of the dual of $\langle e_{i_1},\dots,e_{i_s}\rangle$. This implies that if $q(\beta_1,\dots,\beta_s)$ is a polynomial in the $\beta_j$, we have $i_p^*q(\beta_1,\dots,\beta_s)=q(-t_{i_1},\dots,-t_{i_s})$. Putting all together, we get
\begin{equation}\label{eq:loc}\pi_*q(\beta_1,\dots,\beta_s) = \sum_{(i_1,I)} \frac{q(-t_{i_1},\dots,-t_{i_s})}{\prod (t_{i_m}-t_{i_1}) \prod(t_{i_k}-t_{i_j})}. \end{equation}
This formula can be used to compute the coefficients $\pi_*(C_s(a_1,\dots,a_r)P(\beta_1,b_1,\dots,b_{s-1}))$ appearing in \Cref{thm:chow}.
\subsubsection{Moduli of curves of genus five}
In \cite{PV} the authors computed the Chow ring of moduli spaces of triple covers of $\PP^1$. In particular, their result shows that $\ch(H_{3,5})\simeq \QQ[\lambda_1]/(\lambda_1^3)$, where $H_{3,5}$ is the moduli space of trigonal curves of genus five and $\lambda_1$ is the class of the Hodge line bundle.

The complement of $H_{3,5}$ in the moduli space $M_{5}$ of smooth curves of genus five is isomorphic to the coarse moduli space of $\cM^{\PGL}_4(2,2,2)$ (see \ref{rmk:examples}), whose Chow ring we can compute using \Cref{thm:chow}. This enables us to reprove the following Theorem of Izadi (see \cite{Iza}).
\begin{proposition}[Izadi]\label{prop:chow M5}
Let $M_5$ be the moduli space of smooth curves of genus five. Then
\[ \ch(M_5)\simeq \QQ[\lambda_1]/(\lambda_1^4) \]
where $\lambda_1$ is the first Chern class of the Hodge line bundle.
\end{proposition}
\begin{proof}
\Cref{thm:chow} tells us that
\[\ch(\cM^{\PGL}_4(2,2,2)) \simeq \QQ[a_{1,1},a_{1,2},a_{1,3},c_2,\dots,c_5]/I. \]
In what follows, as generators of the flag variety we use $\beta_1$ and $\sigma_m=\sigma_m(\beta_1,\dots,\beta_s)$ instead of $\beta_1$ and $b_m$.
We use the localization formula (\ref{eq:loc}) to explicitly compute some relations in $I$. These computations are carried out using Mathematica.

In degree one, we have the relation given by ${\rm pr}_{1*}([\widetilde{Z}_{(2,2,2),4}])$ (we already set $c_1=0)$). This gives the relation $40a_{1,1}=0$. In degree two, we have two relations given by ${\rm pr}_{1*}([\widetilde{Z}_{(2,2,2),4}]\cdot{\rm pr}_2^*\beta_1)$ and ${\rm pr}_{1*}([\widetilde{Z}_{(2,2,2),4}]\cdot{\rm pr}_2^*\sigma_1)$. These turn out to be
\begin{align*}
    &48 c_1^2 - 24 c_2 - 56 c_1a_{1,1} + 20 a_{1,1}^2 \\
    &112 c_1^2 - 40 c_2 - 136 c_1 a_{1,1} + 40 a_{1,1}^2 + 20 a_{1,2}
\end{align*}
from which we deduce that $a_{1,2}=c_2=0$. In degree three, we compute ${\rm pr}_{1*}([\widetilde{Z}_{(2,2,2),4}]\cdot{\rm pr}_2^*\xi)$ for $\xi=\sigma_2$, $\beta_1^2$. These are
\begin{align*}
    &-112 c_1^3 + 88 c_1 c_2 - 24 c_3 + (184 c_1^2 - 48 c_2) a_{1,1} \\
    &- 96 c_1 a_{1,1}^2 + 
 12 a_{1,1}^3 - 56 c_1 a_{1,2} + 46 a_{1,1} a_{1,2} - 18 a_{1,3} =0, \\
 &48 c_1^3 + 72 c_1 c_2 - 32 c_3 + (56 c_1^2 - 36 c_2) a_{1,1}\\
 &- 24 c_1 a_{1,1}^2 + 4 a_{1,1}^3 - 4 c_1 a_{1,2} + 2 a_{1,1} a_{1,2} + 4 a_{1,3} = 0.
\end{align*}
Combined with the previous relations, we deduce that $c_3=a_{1,3}=0$.
Next we compute $\pi_*(C_3(0,0,0)\cdot \beta_1^j)$ for $j=3$ and $4$:
\begin{align*}
    &48 c_1^4 - 120 c_1^2 c_2 + 24 c_2^2 + 80 c_1 c_3 - 32 c_4 =0 \\
    &-48 c_1^5 + 168 c_1^3 c_2 - 96 c_1 c_2^2 - 128 c_1^2 c_3 + 56 c_2 c_3 + 
 80 c_1 c_4 - 32 c_5=0.
\end{align*}
As we already know that all the other terms are zero, the first relation implies that $c_4=0$, and similarly the second one implies $c_5=0$. We have proved that $\ch(M_5\smallsetminus H_{3,5})\simeq\QQ$.

We know from \cite{Fab-conj}*{Theorem 2} that $\lambda_1^3\neq 0$ in $\ch(M_5)$. These two facts, combined with the computation of \cite{PV} and the exactness of the localization sequence
\[ \QQ[\lambda_1]/(\lambda_1^3)\simeq\ch(H_{3,5})\longrightarrow\ch(M_5)\longrightarrow\ch(M_5\smallsetminus H_{3,5})\longrightarrow 0 \]
easily implies that $\ch(M_5)\simeq\QQ[\lambda_1]/(\lambda_1^4)$.
\end{proof}
\subsubsection{Moduli of polarized K3 surfaces of degree eight}
Let $K_{8}$ be the moduli space of polarized K3 surfaces of degree eight. There is an open subvariety $U_8\subset K_8$ whose points correspond to polarized K3 surfaces $(X,[L])$ such that $L$ is very ample, and $X$ does not contain any curve of arithmetic genus $1$ and degree $3$. The complement of $U_8$ in $K_8$ is the union of three Noether-Lefschetz divisors, namely $D_{1,1}$, $D_{2,1}$ and $D_{3,1}$, where points in $D_{d,1}$ correspond to polarized K3 surfaces containing a curve of arithmetic genus $1$ and degree $d$.

As observed in \Cref{rmk:examples}, if $(X,[L])$ is a point of $U_8$, then the polarization $L$ embeds $X$ in $\PP^5$ as a complete intersection of three quadrics: in other words, the scheme $U_8$ is isomorphic to the coarse moduli space of $\cM^{\PGL}_5(2,2,2)$, hence we can use \Cref{thm:chow} and \Cref{cor:chow PGL} to compute the Chow ring of $U_8$.
\begin{proposition}\label{prop:chow U8}
We have $\ch(U_8)\simeq \QQ$, hence the pushorward morphism
\[ {\rm CH}^{i-1}(\cup_{d=1}^{3} D_{d,1})\longrightarrow{\rm CH}^i(K_8) \]
from the union of these Noether-Lefschetz divisors is surjective in degree $i>0$.
\end{proposition}
\begin{proof}
In what follows, we adopt the same notation used in the proof of \Cref{prop:chow M5}. We know from \Cref{cor:chow PGL} that
\[ \ch(U_8)\simeq \QQ[a_{1,1},a_{1,2},a_{1,3},c_1,c_2,\dots,c_6]/(c_1,I). \]
In degree one we have the single relation given by $\pr_{1*}[\widetilde{Z}_{(2,2,2),5}]$: this is equal to $80a_{1,1}$, hence $a_{1,1}=0$ in the rational Chow ring. 

Degree two relations are given by $\pr_{1*}([\widetilde{Z}_{(2,2,2),5}]\cdot\xi)$ for $\xi=\beta_1$ or $\sigma_1$, and usual computations with localization formulas give explicit expressions for these relations:
\begin{align*}
    &80 c_1^2 - 32 c_2 - 120 c_1 a_{1,1} + 60 a_{1,1}^2 - 20 a_{1,2}=0\\
    &200 c_1^2 - 56 c_2 - 300 c_1 a_{1,1} + 120 a_{1,1}^2 + 10a_{1,2}=0.
\end{align*}
Together, they imply that $a_{1,2}=c_2=0$. Next, we compute the degree three relations given by $\pr_{1*}([\widetilde{Z}_{(2,2,2),5}]\cdot\xi)$ for $\xi=\beta_1^2$ and $\sigma_2$. They are
\begin{align*}
    &-80 c_1^3 + 112 c_1 c_2 - 56 c_3 + (120 c_1^2 - 64 c_2) a_{1,1} \\
    &- 80 c_1 a_{1,1}^2 +  24 a_{1,1}^3 + 20 c_1 a_{1,2} - 18 a_{1,1} a_{1,2} + 14 a_{1,3}=0 \\
    &-240 c_1^3 + 168 c_1 c_2 - 72 c_3 + (480 c_1^2 - 92 c_2) a_{1,1}\\
    &- 340 c_1 a_{1,1}^2 + 72 a_{1,1}^3 - 80 c_1 a_{1,2} + 96 a_{1,1} a_{1,2} - 63 a_{1,3}=0.
\end{align*}
These relations, combined with the previous ones, show that $a_{1,3}=c_3=0$ in the Chow ring. To show that $c_4=0$, it is enough to prove that $\pi_*(C_4(0,0,0)\cdot\beta_1^3)$ is not zero, because we already know that all the other terms appearing in the relation given by ${\rm pr}_{1*}([\widetilde{Z}_{(2,2,2),5}]\cdot \beta_1^3)$ are zero. After straightforward computations we get
\[ 80 c_1^4 - 192 c_1^2 c_2 + 32 c_2^2 + 136 c_1 c_3 - 48 c_4=0, \]
hence $c_4=0$. In the same way, the coefficient $\pi_*(C_4(0,0,0)\cdot\beta_1\sigma_3)$ turns out to be
\[ -200 c_1^5 + 512 c_1^3 c_2 - 136 c_1 c_2^2 - 464 c_1^2 c_3 + 64 c_2 c_3 + 
 272 c_1 c_4 - 48 c_5=0, \]
hence $c_5=0$. Finally, we compute $\pi_*(C_4(0,0,0)\cdot\beta_1^2\sigma_3)$ and we get
\begin{align*}
    &200 c_1^6 - 712 c_1^4 c_2 + 448 c_1^2 c_2^2 - 24 c_2^3 + 664 c_1^3 c_3 - 440 c_1 c_2 c_3\\
    &+ 88 c_3^2 - 472 c_1^2 c_4+ 96 c_2 c_4 + 200 c_1 c_5 - 48 c_6=0
\end{align*}
which implies $c_6=0$ and concludes the proof.
\end{proof}
\section{Integral Picard groups}\label{sec:pic}
In \cite{Ben-deg} Benoist gives beautiful formulas for the multidegree of the divisor $S_{\bfd,n}$ of singular complete intersections in the Hilbert scheme $\hilb_{\bfd,n}\simeq\Gr(r_{\ell},\cE_{\ell})$. This is equivalent to compute the integral Picard group of $\cM^{\SL}_n(\bfd)$. In this Section we leverage \Cref{thm:chow} to compute $\Pic(\cM^{\SL}_n(\bfd))$, thus giving a new proof of Benoist's formula.

\Cref{thm:pic pgl} gives a presentation for the integral Picard group of $\cM^{\PGL}_n(\bfd)$, the stack of smooth complete intersections in $\PP^n$. This result, specialized to the case $\bfd=(2,2)$, recovers \cite{AI}*{Theorem 1.1} (see \Cref{cor:pic ddd}).

Let us recall some notation from \Cref{thm:chow}: given $\bfd=(d_1,\dots,d_r)$, there are integers $r_1,\dots,r_\ell$ such that
\begin{align*}
    & d_1=\cdots=d_{r_1},\\
    & d_{r_1+1}=\cdots=d_{r_1+r_2},\\
    & \vdots \\
    & d_{r_1+r_2+\cdots + r_{\ell-1}+1} = \cdots = d_{r_1+\cdots +r_\ell}.
\end{align*}
Given symbols $\gamma_1,\dots,\gamma_r$, we can subdivide them into $\ell$ subsets of the form $S_j=\{\gamma_{r_1+\cdots r_{j-1}+1},\dots,\gamma_{r_1+\cdots+r_j}\}$. The symmetric group $\mathfrak{S}_{r_j}$ acts on this subset, and we denote $a_{j,k}$ for $k=1,\dots,r_j$ the elementary symmetric functions with variables in the set $S_j$.

If we assume that the base field $k$ has characteristic $\neq 2$ or that $n$ is odd, we have that the pushforward of $[\widetilde{Z}_{\bfd,n}]$ is equal to the fundamental class of $Z_{\bfd,n}$; otherwise, it is two times the fundamental class (see \cite{Ben-deg}*{Proof of Proposition 4.2} and the references contained therein). Then the following is a straightforward consequence of \Cref{thm:chow}.
\begin{proposition}\label{prop:rel Pic}
Suppose that the base field $k$ has characteristic $\neq 2$ or that $n$ is odd. Then the integral Picard group of $\cM^{\SL}_n(\bfd)$ is generated by the set $\{a_{1,1},\dots,a_{\ell,1}\}$ modulo the single relation
\begin{align*}
    \sum_{i=1}^{r} \gamma_{i}\cdot \pi_* [ & \sigma_{s-1}((d_i-1)\beta_1+\beta_1,\dots,(d_i-1)\beta_1+\beta_s)\\
    &\cdot\prod_{j=1,j\neq i}^{r} \sigma_{s}((d_j-1)\beta_1+\beta_1,\dots,(d_j-1)\beta_1+\beta_s) ] =0
\end{align*}
where $\pi_*:\ch_{\PGL_{n+1}}({\rm Fl}_{n+1,s})\to\ch_{\PGL_{n+1}}(\spec{k})$ is the pushforward morphism.
\end{proposition}
From now on, we will write $e_i:=d_i-1$. Let us compute more explicitly the coefficient in front of $\gamma_i$ in the relation appearing in \Cref{prop:rel Pic}. First observe that
\begin{align*}
    \sigma_s(e_j\beta_1+\beta_1,e_j\beta_1+\beta_2,\dots,e_j\beta_1+\beta_s)&=\prod_{h=1}^{s} (e_j\beta_1+\beta_h)\\
    &=\sum_{k_j=0}^{s} (e_j\beta_1)^{k_j}\sigma_{s-k_j}(\beta_1,\dots,\beta_s)
\end{align*}
and that
\begin{align*}
    \sigma_{s-1}(e_i\beta_1+\beta_1,e_i\beta_1+\beta_2,\dots,e_i\beta_1+\beta_s)&=\sum_{h'=1}^{r}\left(\prod_{h=1,h\neq h'}^{r}(e_i\beta_1+\beta_h) \right)\\
    &=\sum_{k_1=0}^{s-1}(k_1+1)(e_i\beta_1)^{k_1}\sigma_{s-1-k_1}(\beta_1,\dots,\beta_s).
\end{align*}
We deduce that the coefficient in front of $\gamma_i$ can be rewritten as
\begin{equation}\label{eq:coefficient 1}
     \sum_{k_i\in [s-1], k_j\in [s]} (k_i+1)e_1^{k_1}\cdots e_r^{k_r}\cdot \pi_*\left(\beta_1^{\sum k_j}(\sigma_{s-k_1}\sigma_{s-k_2}\cdots\sigma_{s-1-k_i}\cdots\sigma_{s-k_r})\right).
\end{equation} 
A priori, in the formula above we should sum over all the possible values of $k_j$, but it turns out that many terms are zero, as the next Lemma states.
\begin{lemma}\label{lm:some terms are zero}
The terms in the sum of (\ref{eq:coefficient 1}) are zero for $\sum_{j=1}^{r} k_j<s-1$ and $\sum_{j=1}^{r} k_j>n$.
\end{lemma}
\begin{proof}
First observe that the flag variety ${\rm Fl}_{n+1,s}$ is isomorphic to the projective bundle $\PP(\cT)\to\Gr(s,n+1)$, where $\cT$ denotes the tautological vector bundle of rank $s$. In particular, we have a factorization of $\pi$ as
\[ \PP(\cT)\overset{p}{\longrightarrow}\Gr(s,n+1)\overset{q}{\longrightarrow}\spec{k}. \]
Observe moreover that the $\sigma_m$ appearing in (\ref{eq:coefficient 1}) are the Chern classes of $\cT^{\vee}$, and $\beta_1$ is the hyperplane section of $\PP(\cT)$. In particular we get
\begin{align*}
    \pi_*\left(\beta_1^{\sum k_j}(\sigma_{s-1-k_1}\sigma_{s-k_2}\cdots\sigma_{s-k_r})\right)=q_*\left(p_*(\beta_1^{\sum k_j})\cdot \sigma_{s-1-k_1}\sigma_{s-k_2}\cdots\sigma_{s-k_r}\right),
\end{align*}  
and for $\sum k_j<s-1$ we have $p_*(\beta_1^{\sum k_j})=0$. This proves the first part of the Lemma.

For the second part, we have by definition that $p_*\beta_1^{d}=s_{d-{s-1}}(\cT)$, the $d^{\rm th}$ equivariant Segre class of the tautological subbundle. As we already know that the term on the right in the formula above belongs to ${\rm CH}^0_{\SL_{n+1}}(\spec{k})\simeq {\rm CH}^0(\spec{k})$, we can compute the pushforward in the non-equivariant setting.

Recall that in the (non-equivariant) Chow ring of $\Gr(s,n+1)$ we have the relation $c(\cT)c(\cQ)=1$, where $\cQ$ is the tautological quotient bundle. Using the fact that the total Segre class is the inverse of the total Chern class, we deduce that $s(\cT) = c(\cQ)$. As $c_d(\cQ)=0$ for $d>n-s+1$, we deduce that $p_*\beta_1^d=0$ for $d-s+1>n-s+1$, as claimed.
\end{proof}

\begin{lemma}\label{lm:schubert}
Let $q:\Gr(s,n+1)\to\spec{k}$ be the projection map and set $d= \sum_{j=1}^{r} k_j$. Suppose that $s-1\leq d \leq n$, then
\[ q_*\left(c_{d-s+1}(\cQ)\cdot \sigma_{s-1-k_1}\sigma_{s-k_2}\cdots\sigma_{s-k_r}\right)=1, \]
where $\cQ$ is the tautological quotient bundle and $\sigma_m=c_m(\cT^{\vee})$.
\end{lemma}
\begin{proof}
We are going to apply some basic facts of Schubert calculus. Let us first consider the case $d=s-1$: we have to prove that $q_*(\sigma_{s-1-k_1}\sigma_{s-k_2}\cdots\sigma_{s-k_r})=1$.  The classes $\sigma_m$ correspond to the Schubert cycles $\sigma_{(1,\dots,1)}$, where $(1,\dots,1)=(1^m)$ should be thought as the Young diagram with one column and $m$ rows, and ${\rm CH}^{s(n+1-s)}(\Gr(s,n+1))$ is generated by the cycle $\sigma_{(r-1,\dots,r-1)}$, whose associated Young diagram is a rectangle with $s$ rows and $r-1=n+1-s$ columns; this is the only Young diagram with $s(r-1)$ squares whose associated Schubert class is not zero.

The product of a Schubert cycle $\sigma_{\lambda}$ by $\sigma_m=\sigma_{(1^m)}$ can be computed using Pieri's formula: this tells us that 
\[ \sigma_\lambda\cdot\sigma_{(1^m)} = \sum \sigma_{\mu} \]
where the sum is taken over all the Young diagrams $\mu$ that can be obtained from $\lambda$ by adding $m$ squares, with the rule that one can add at most one square per row. This rule can be used to compute the product $\sigma_{s-1-k_1}\sigma_{s-k_2}\cdots\sigma_{s-k_r}$.

Indeed, this product will be a sum of Schubert cycles associated to Young diagrams having $s(r-1)$ squares, so to actually compute it we only have to count how many times the cycle $\sigma_{(r-1,\dots,r-1)}$ appears, as the other Young diagrams of the same dimension yield cycles that are zero in the Chow ring. 

This can be rephrased as follows: take a rectangle with $r-1$ columns and $s$ rows, and tick $s-1-k_1$ squares in the first column; we want to count the number of ways in which we can tick the whole rectangle with $r-1$ moves, each move consisting of ticking $s-k_j$ squares in such a way that at each step the ticked diagram is a Young diagram, and no more than one new square per row has been ticked (Pieri's rule). Then our claim is that it exists exactly one way to do so.

To prove existence, consider the following set of moves: each time, we tick all the squares that are below the last ticked square; if we finish the column, we move to the next column, starting from the top square and going down. In this way we are following the rules given by Pieri's formula, because to tick two squares in the same row in the same move we would need to tick at least $s+1$ squares, which never happens. As 
\[ \left(\sum_{j=1}^{r} s-k_j \right) - 1= rs-1-\sum_{j=1}^{r} k_j = rs - 1 -s +1 =(r-1)s, \]
we will end up ticking the whole rectangle.

To show uniqueness, observe that in each move the number of columns completely ticked can raise of at most one. We only have $r-1$ moves at our disposal, and we start with zero columns completely ticked, because $s-1-k_1<s$. This means that at each step we have to finish exactly one column, and the only way to do so by following Pieri's rule is by following the algorithm described before.

Putting all together, this shows that $\sigma_{s-1-k_1}\sigma_{s-k_2}\cdots\sigma_{s-k_r}=\sigma_{((r-1)^s)}$, hence its pushforward along $\Gr(s,n+1)\to\spec{k}$ is equal to $1$.

The proof in the case $d=\sum_{j=1}^{r} k_j>(s-1)$ proceeds along almost the same lines: the only difference is in the fact that instead of ticking all the squares in a rectangle, we have to tick all the squares in the Young diagram obtained by removing $d - s +1$ squares from the last row of the $s\times (r-1)$-rectangle; indeed, the Schubert class associated to this Young diagram is the only class that paired with $c_{d-s+1}(\cQ)=\sigma_{(d-s+1)}$ is not zero.

Adapting the argument used before, we conclude that there exists a unique way to tick this Young diagram following Pieri's rule, from which we get the desired conclusion.
\end{proof}

\begin{lemma}\label{lm:localization}
Let $e_1,\dots, e_r$ be integers $\geq 0$, and set $d_i=e_i-1$ and $s=n-r+2$. Then for every $i=1,\dots,r$ the following equality holds:
\[ d_1d_2\cdots \widehat{d_i}\cdots d_r\sum_{j=1}^{r}\frac{1}{\prod_{j'\neq j} e_j-e_{j'}}\left(\frac{e_i^{n+1}-e_j^{n+1}}{e_i-e_j}\right) =  \sum_{k_i\in [s-1], k_j\in [s]} (k_i+1)e_1^{k_1}\cdots e_r^{k_r}  \]
where on the right the summation is over the $k_1,\dots,k_r$ with $s-1 \leq \sum k_j \leq n$.
\end{lemma}
\begin{proof}
First recall the following easy polynomial identity:
\[ \frac{t_i^{n+1}-t_j^{n+1}}{t_i-t_j} = \sum_{k=0}^{n} t_i^{n-k}t_j^k. \]
Let $V^{\vee}$ be the dual of the standard representation of the torus $T=\gm^{\oplus r}$. The fixed points of the $T$-action on $\PP(V^{\vee})$ are those points $p_1,\dots, p_r$ where only one of the homogeneous coordinates $x_1,\dots, x_r$ is non-zero. A basis for the tangent space of $\PP(V^{\vee})$ at $p_j$ is given by the elements of the form $(x_{j'}/x_j)^{\vee}$, on which $T$ acts via the character $t_j-t_{j'}$ (here the $t_j$ are by definition the characters of the standard representation). In particular, we deduce that
\[ c_r^T(T\PP(V^{\vee})_{p_j}) = \prod_{j'\neq j} (t_j-t_{j'}). \]
If $h$ dentotes the hyperplane class in $\ch_T(\PP(V^{\vee}))$, then the restriction of $h$ to $\ch_T(p_j)$ is equal to $t_j$, because the rank one representation $\cO(1)|_{p_j}$ is generated by $x_j$.

Let $\pi:\PP(V^{\vee})\to\spec{k}$ be the natural projection. It follows then from the localization formula (\cite{EG-loc}*{Theorem 2}) that
\[ \pi_*\left(\sum_{k=0}^{n} t_i^{n-k}h^k\right) = \sum_{j=1}^{r} \left( \frac{\sum_{k=0}^{n} t_i^{n-k}t_j^k}{\prod_{j'\neq j} (t_j-t_{j'})} \right). \]
We can also compute the term on the left directly: indeed, by definition $\pi_*h^{k}$ is equal to the equivariant Segre class $s_{k-s+1}^T(V^{\vee})$. Recall that the total Segre class is the inverse of the total Chern class. In our case we have:
\begin{align*}
    s^T(V^{\vee})&=\left( c^T(V^{\vee})\right)^{-1} = \left( \prod_{j=1}^{r} (1-t_j) \right)^{-1}\\
    &=\prod_{j=1}^{r}\left(\sum_{k_j\geq 0} t_j^{k_j} \right) = \sum_{k_j\geq 0} t_1^{k_1}\cdots t_r^{k_r}
\end{align*}
This shows that
\begin{align*}
    \pi_*\left(\sum_{k=0}^{n} t_i^{n-k}h^k \right)&=\sum_{k=s-1}^{n} t_i^{n-k}\left(\sum_{k_1+\cdots +k_r=k-s+1} t_1^{k_1}\cdots t_r^{k_r}\right) \\
    &= \sum_{k_1+\cdots +k_r=s-1} (k_i+1)t_1^{k_1}\cdots t_r^{k_r}.
\end{align*}
An obvious but important remark is that the $k_j$ in the sum above goes from $0$ to $s-1$. Putting all together, this shows that the following polynomial identity holds:
\begin{equation}\label{eq:id 1} \sum_{j=1}^{r} \frac{t_i^{n+1}-t_j^{n+1}}{t_i-t_j} = \sum_{k_1+\cdots +k_r=s-1} (k_i+1)t_1^{k_1}\cdots t_r^{k_r}.  \end{equation}
If we multiply the term on the left by $\prod_{i'\neq i} (t_{i'}+1)$ and we evaluate in $e_1,\dots,e_r$, we get the left hand side of the formula that appears in the statement of the Lemma. Hence, let us multiply also the right hand side of (\ref{eq:id 1}) by this factor:
\begin{align*}
    \prod_{i'\neq i} (t_{i'}+1) \cdot \sum_{k_1+\cdots +k_r=s-1} (k_i+1)t_1^{k_1}\cdots t_r^{k_r} = \sum_{s-1\leq \sum k_j \leq s+r-2} (k_i+1) t_1^{k_1}\cdots t_r^{k_r}
\end{align*}
where the $k_j$ for $j\neq i$ now range from $0$ to $s$, and $k_i$ still goes from $0$ to $s-1$. Observe moreover that $s+r-2=n$; evaluating this polynomial in the $e_1,\dots,e_r$, we get the claimed identity.
\end{proof}
We now have all the ingredients necessary to compute the integral Picard group of $\cM^{\SL}_n(\bfd)$. This also gives a new proof of Benoist's formulas.
\begin{theorem}[\cite{Ben-deg}*{Theorem 1.3}]\label{thm:pic}
Suppose that the base field $k$ has characteristic $\neq 2$ or that $n$ is odd. Set $a_{j,1}=\gamma_{r_1+\cdots+r_{j-1}+1}+\cdots+\gamma_{r_1+\cdots+r_j}$. Then
\[ \Pic(\cM^{\SL}_n(\bfd))\simeq\langle a_{1,1},\dots,a_{\ell,1}\rangle/\langle F \rangle \]
where 
\[F= \sum_{i=1}^{r} \left( d_1d_2\cdots \widehat{d_i}\cdots d_r\sum_{j=1}^{r}\frac{1}{\prod_{j'\neq j} e_j-e_{j'}}\left(\frac{e_i^{n+1}-e_j^{n+1}}{e_i-e_j}\right)\right)\gamma_i \]
\end{theorem}
\begin{proof}
We computed in (\ref{eq:coefficient 1}) a first expression for the coefficient in front of $\gamma_i$ inside the relation of degree one given by $\pr_{1*}([\widetilde{Z}_{\bfd,n}])$. This can be simplified thanks to \Cref{lm:some terms are zero} and \Cref{lm:schubert}. We deduce that
$\Pic(\cM^{\SL}_n(\bfd))$ is generated by the symmetric elements $\langle a_{1,1},\dots,a_{\ell,1} \rangle$ modulo the relation
\begin{align*}
    \sum_{i=1}^{r}  \left(\sum_{k_i\in [s-1], k_j\in [s]} (k_i+1)e_1^{k_1}\cdots e_r^{k_r}\right) \gamma_i = 0
\end{align*}
where on the right the summation is over the $k_1,\dots,k_r$ with $s-1 \leq \sum k_j \leq n$. \Cref{lm:localization} shows that the sums appearing above coincide with the ones given in terms of $e_1,\dots,e_r$ that appear in the statement of the Theorem.
\end{proof}

\subsection{Integral Picard group of $\cM^{\PGL}_n(\bfd)$}
From \Cref{thm:pic} we can deduce a description of the integral Picard group of $\cM^{\PGL}_n(\bfd)$. First, the integral Picard group of the tower of Grassmannian bundles 
\[ \Gr(r_{\ell},\cE_{\ell}) \overset{\pi_{\ell}}{\to} \cdots \overset{\pi_3}{\to}\Gr(r_{2},\cE_{2}) \overset{\pi_2}{\to}\Gr(r_{1},\cE_{1})\overset{\pi_1}{\to} \spec{k} \]
can be identified with the free abelian group
\begin{equation*}
    \ZZ^{\oplus \ell}\simeq \oplus_{i=1}^{\ell} \ZZ\cdot[\det(\cT_i)]
\end{equation*}
where $\cT_i$ is the pullback to $\Gr(r_{\ell},\cE_{\ell})$ of the tautological bundle of $\Gr(r_i,\cE_i)$. Second, define $d'_i:=d_{r_1+\cdots +r_i}$ and let $(w_1,\dots,w_\ell)$ be a tuple such that
\begin{equation*}
    \sum_{i=1}^{\ell} w_ir_id'_i={\rm{gcd}}(r_1d'_1,\dots,r_\ell d'_{\ell}).
\end{equation*}
and set
\[u=\frac{{\rm mcm}(n+1, {\rm{gcd}}(r_1d'_1,\dots,r_\ell d'_{\ell}))}{{\rm{gcd}}(r_1d'_1,\dots,r_\ell d'_{\ell})} \]
Third, define $\Lambda$ as the kernel of the homomorphism
\begin{equation*}
    \ZZ^{\oplus \ell}\longmapsto\ZZ,\quad (x_1,\dots,x_\ell)\longmapsto \sum_{i=1}^{\ell} x_ir_id'_i.
\end{equation*}
Finally, let $F\in\ZZ^{\oplus\ell}$ be the element whose $i^{\rm th}$-entry is
\begin{equation*}
     d_1d_2\cdots \widehat{d'_i}\cdots d_r\sum_{j=1}^{r}\frac{1}{\prod_{j'\neq j} e_j-e_{j'}}\left(\frac{e_i^{n+1}-e_j^{n+1}}{e_i-e_j}\right).
\end{equation*}
We have all the elements necessary to describe the integral Picard group of $\cM^{\PGL}_n(\bfd)$.
\begin{theorem}\label{thm:pic pgl}
There is an injective homomorphism
\[\Pic(\cM^{\PGL}_n(\bfd))\longrightarrow\Pic(\cM^{\SL}_n(\bfd))\simeq\ZZ^{\oplus \ell}/\langle F \rangle \]
whose image contains $F$, and induces an isomorphism
\[ \Pic(\cM^{\PGL}_n(\bfd)) \simeq \Lambda\oplus\langle (uw_1,\dots,uw_\ell) \rangle / \langle F \rangle \]
\end{theorem}
\begin{proof}
For every $\PGL_{n+1}$-scheme $X$, the homomorphism of algebraic groups $\SL_{n+1}\to\PGL_{n+1}$ induces a morphism of quotient stacks $[X/\SL_{n+1}]\to [X/\PGL_{n+1}]$. Applying this to $X=\hilb_{\bfd,n}^{\rm sm}$ we get
\[ \hilb_{\bfd,n}^{\rm sm} \longrightarrow [\hilb_{\bfd,n}^{\rm sm}/\SL_{n+1}]\simeq \cM^{\SL}_n(\bfd)\longrightarrow [\hilb_{\bfd,n}^{\rm sm}/\PGL_{n+1}]\simeq \cM^{\PGL}_n(\bfd). \]
We can pull back line bundles along this composition, obtaining homomorphisms
\[ \Pic^{\PGL_{n+1}}(\hilb_{\bfd,n}^{\rm sm})\longrightarrow \Pic^{\SL_{n+1}}(\hilb_{\bfd,n}^{\rm sm}) \longrightarrow \Pic(\hilb_{\bfd,n}^{\rm sm}). \]
This composition is injective because its kernel is isomorphic to the group of characters of $\PGL_{n+1}$, which is trivial. This implies that the first map is also injective.

The second map is injective for the same reason, and it is also surjective because the line bundles $\det(\cT_j)$ all admit a $\SL_{n+1}$-linearization, as it is already clear from \Cref{thm:pic}. This implies that the image of the pullback along the first map can be identified with the image of the pullback along the composition. 

Consider the commutative square of pullbacks
\begin{equation}\label{eq:diag pic}
\begin{tikzcd}
\Pic^{\PGL_{n+1}}(\Gr(r_{\ell},\cE_{\ell})) \ar[r, "\varphi"] \ar[d] & \Pic(\Gr(r_{\ell},\cE_{\ell})) \ar[d] \\
\Pic^{\PGL_{n+1}}(\hilb_{\bfd,n}^{\rm sm})\ar[r,"\psi"] & \Pic(\hilb_{\bfd,n}^{\rm sm}).
\end{tikzcd}
\end{equation}
Observe that the element $F$ belongs to the image of $\varphi$ because the discriminant divisor $S_{\bfd,n}$ is invariant with respect to the $\PGL_{n+1}$-action.
We deduce that the image of $\psi$ is equal to the image of $\varphi$ modulo $F$, or in other terms that the image of the pullback is equal to the subgroup of $\Pic(\Gr(r_{\ell},\cE_{\ell}))$ of line bundles admitting a $\PGL_{n+1}$-linearization, modulo $F$.
A line bundle in $\Pic(\Gr(r_{\ell},\cE_{\ell}))$ is of the form 
\[ \cL=\det(\cT_1)^{\otimes k_1} \otimes \det(\cT_2)^{\otimes k_2} \otimes \cdots \otimes \det(\cT_\ell)^{\otimes k_\ell}.\]

The points in the total space of $\det(\cT_i)^{\otimes c_i}$ are given by pairs $((f_1,\dots,f_{r_i}),(f_1\wedge\dots\wedge f_{r_i})^{\otimes k_i})$, where the $f_j$ are linearly independent homogeneous forms of degree $d'_i$.
Any $\GL_{n+1}$-linearization of $\det(\cT_i)$ is of the following form: given an element $A$ of $\GL_{n+1}$, it acts on a point in the total space by sending 
\[(f_1,\dots,f_{r_i})\longmapsto (f_1(A^{-1},\underline{x})\dots,f_{r_i}(A^{-1}\underline{x}))\] and
\[ A\cdot (f_1\wedge\dots\wedge f_{r_i})^{\otimes k_i}= \det(A)^{p_i} (f_1(A^{-1}\underline{x})\wedge \dots \wedge f_{r_i}(A^{-1}\underline{x}))^{\otimes k_i}   \]
The subtorus $\gm\subset \GL_{n+1}$ of scalar matrices acts as
\[ \lambda \cdot  (f_1\wedge\dots\wedge f_{r_i})^{\otimes k_i}= \lambda^{p_i(n+1)}\cdot \lambda^{-r_id'_ik_i}\cdot  (f_1\wedge\dots\wedge f_{r_i})^{\otimes k_i}. \]
From this we see that the subtorus $\gm$ acts on $\cL$ with weight
\[ (n+1)\sum_{i=1}^{\ell} p_i - \sum_{i=1}^{\ell} (r_id'_i)k_i. \]
For a given $(k_1,\cdots,k_\ell)$ the character above is trivial if and only if $n+1$ divides $\sum_{i=1}^{\ell} (r_id'_i)k_i$, hence the subgroup of line bundles admitting a $\PGL_{n+1}$-linearization can be identified with the the preimage of $(n+1)\ZZ$ along the homomorphism
\begin{equation}\label{eq:map}
    \ZZ^{\oplus \ell}\longmapsto\ZZ,\quad (x_1,\dots,x_\ell)\longmapsto \sum_{i=1}^{\ell} x_ir_id'_i.
\end{equation}  
The element $(w_1,\dots,w_{\ell})$ is sent to ${\rm{gcd}}(r_1d'_1,\dots,r_\ell d'_{\ell})$, which is also the generator of the image as a subgroup. This implies that $(uw_1,\dots,uw_{\ell})$ is sent to ${\rm mcm}(n+1, {\rm{gcd}}(r_1d'_1,\dots,r_\ell d'_{\ell}))$ and that the subgroup generated by $(uw_1,\dots,uw_{\ell})$ surjects onto the intersection of the image with $(n+1)\ZZ$. This shows that the preimage of $(n+1)\ZZ$, which coincides with the image of $\varphi$ in (\ref{eq:diag pic}), is isomorphic to the sum of the subgroup generated by $(uw_1,\dots,uw_{\ell})$ and the kernel of (\ref{eq:map}). As the image of $\psi$ is equal to the image of $\varphi$ modulo $F$, this concludes the proof.
\end{proof}
If we specialize the Theorem above to the case of complete intersections of codimension $r$ and type $\bfd=(d,\dots,d)$, we obtain the following.
\begin{corollary}\label{cor:pic ddd}
Suppose that the base field $k$ has characteristic $\neq 2$ or that $n$ is odd. Then we have
\[ \Pic(\cM^{\PGL}_n(d,\dots,d)) \simeq \ZZ/N\ZZ,\quad N=\frac{\binom{n+1}{r}rd^{r}(d-1)^{n-r+1}}{{\rm mcm}(n+1,rd)}. \]
\end{corollary}
In particular, for $d=r=2$, we recover \cite{AI}*{Theorem 1.1}.

\section{The codimension two case}\label{sec:cod 2}
In this Section, we compute explicitly the Chow ring of moduli of smooth complete intersections of codimension 2 (\Cref{thm:chow cod 2 first} and \Cref{thm:chow cod 2 second}). 

We give two applications of this result: in the first one, we give a quick proof of Faber's result on the Chow ring of $M_4$, the moduli space of smooth curves of genus four (\Cref{cor:chow M4}). In the second one, we compute the Chow ring of an open subset of $K_6$, the moduli space of polarized K3 surfaces of degree six (\Cref{cor:chow U6}).

All the Chern (resp. Segre) classes of the equivariant vector bundles appearing in this Section are intended to be equivariant Chern (resp. Segre) classes. In particular, we will use the writing $c_i(E)$ to denote the Chern class of degree $i$ of an equivariant vector bundle $E\to X$, instead of the more correct but notationally heavier version $c_i^G(E)$.
\subsection{Intersection theory on ${\rm Fl}_{n,n+1}$}
As $r=2$ we have $s=n$. The flag variety ${\rm Fl}_{n,n+1}$ is a projective bundle over a projective space. Indeed, we have $\Gr(n,n+1)\simeq \bP^n$, where $\bP^n$ stands for the projectivization of the dual of the standard representation of $\GL_{n+1}$, and ${\rm Fl}_{n,n+1}\simeq\PP(\cT)$, the projectivization of the tautological bundle over $\bP^n$. It follows from the dualized Euler exact sequence that $\cT\simeq \Omega_{\bP^n}(1)$, hence we have ${\rm Fl}_{n+1,n}\simeq \PP(\Omega_{\bP^n}(1))$.

In particular, the $\GL_{n+1}$-equivariant Chow ring of ${\rm Fl}_{n,n+1}$ admits the following presentation
\begin{equation*}
    \ch_{\GL_{n+1}}({\rm Fl}_{n,n+1})\simeq \QQ[\beta_1,\xi_1,c_1,\dots,c_{n+1}]/I.
\end{equation*}
The cycle $\beta_1$ is the hyperplane class of $\PP(\Omega_{\bP^n}(1))$, and $\xi_1$ is the hyperplane class of $\bP^n$, which coincides with the first Chern class of the tautological quotient bundle of $\Gr(n,n+1)$. The ideal of relations $I$ is generated by the two polynomials
\begin{align*}
    &\xi_1^{n+1} - c_1\xi_1^{n} + c_2\xi_1^{n-1} + \cdots + (-1)^{n+1}c_{n+1}, \\
    &\beta_1^{n} + c_1(\Omega_{\bP^n}(1))\beta_1^{n-1} + c_2(\Omega_{\bP^n}(1))\beta_1^{n-2} + \cdots + c_n(\Omega_{\bP^n}(1)).
\end{align*}
The second polynomial can be made more explicit: we have $c(\Omega_{\bP^n}(1))c(\cO(1)) = c(V)$, hence
\begin{align*}
    c(\Omega_{\bP^n}(1)) &= (1+c_1+\cdots +c_{n+1})(1+\xi_1)^{-1} \\
    &= \left( \sum_{j=0}^{n+1} c_j \right) \left( \sum_{i\geq 0} (-1)^i\xi_1^i \right) \\
    &= \sum_{\substack{i\geq 0,\\ 0\leq j \leq n+1}} (-1)^i\xi_1^ic_j.
\end{align*}
This implies that
\begin{equation*}
     c_m(\Omega_{\bP^n}(1)) = \sum_{i=0}^m (-1)^i \xi_1^i c_{m-i}.
\end{equation*}
We will also need an explicit expression of the pushforward of $\beta_1^a\xi_1^b$ along the $\GL_{n+1}$-equivariant pushforward $\pi:{\rm Fl}_{n,n+1}\to \spec{k}$. Using the factorization
\[ \PP(\Omega_{\bP^n}(1)) \overset{p}{\longrightarrow} \bP^n \overset{q}{\longrightarrow} \spec{k}  \]
we get that $\pi_*(\beta_1^a\xi_1^b) = q_*(p_*(\beta_1^a)\cdot \xi_1^b)$.

Let $E\to X$ be an equivariant vector bundle. The total equivariant Segre class $s(E)=1+s_1(E) + s_2(E) + \dots$ is defined as the formal inverse of the total equivariant Chern class $c(E):=1+c_1(E) + c_2(E) + \dots$. If $\pi:\PP(E)\to X$ is the associated projective bundle, and $h$ is the hyperplane class of $\PP(E)$, we have $s_i(E) = \pi_*h^{{\rm rk}(E)+i-1}$.

In particular, we have that $p_*\beta_1^a = s_{a-n+1}(\Omega_{\bP^n}(1))$ and $q_*\xi_1^b = s_{b-n}(V^{\vee})$. In the equivariant Chow ring of $\bP^n=\PP(V^{\vee})$ we have the relation $c(\Omega_{\bP^n}(1))c(\cO(1))=c(V)$, which implies that $s_i (\Omega_{\bP^n}(1)) = s_i (V) + s_{i-1}(V)\xi_1$. We deduce
\begin{align*}
    \pi_*(\beta_1^a\xi_1^b) &= s_{a-n+1}(V)q_*(\xi_1^{b}) + s_{a-n}(V) q_*(\xi_1^{b+1})\\
    &= s_{a-n+1}(V)s_{b-n}(V^{\vee}) + s_{a-n}(V) s_{b-n+1}(V^{\vee}) \\
\end{align*} 
\subsection{Preliminary results}
Let $\bfd=(d_1,d_2)$. From \Cref{thm:chow} we know that
\[  \ch(\cM^{\GL}_n(\bfd))\simeq \QQ[c_1,\dots,c_{n+1},\gamma_1,\gamma_2]^{\mathfrak{S}_{\bfd}}/R   \]
where $\mathfrak{S}_\bfd=\mathfrak{S}_2$ if $d_1=d_2$ and it is trivial otherwise, and the ideal of relations $R$ is generated by cycles of the form
\[ \sum_{0\leq k_1,k_2 \leq n} \gamma_1^{k_1}\gamma_2^{k_2}\cdot \pi_*\left( C_n(k_1,k_2)\beta_1^a \xi_1^b \right) \]
for $0\leq a \leq n-1$ and $0 \leq b \leq n$, and
\begin{align*} C_n(k_1,k_2)= &\sigma_{n-k_1}((d_1-1)\beta_1+\beta_1,\dots,(d_1-1)\beta_1+\beta_n) \\ &\cdot \sigma_{n-k_2}((d_2-1)\beta_1+\beta_1,\dots,(d_2-1)\beta_1+\beta_n).
\end{align*}
To write down more explicit relations, we need to compute the symmetric polynomials in $(d_i-1)\beta_1+\beta_j$ in terms of $\beta_1$ and $\xi_1$. Recall that by definition the $\beta_1,\dots,\beta_n$ are the Chern roots of the dual of the tautological bundle on $\Gr(n,n+1)\simeq\bP^n$. The class $\beta_1$ lives in the Chow ring of the flag variety ${\rm Fl}_{n,n+1}$, and it coincides with the hyperplane class of $\PP(\Omega_{\bP^n}(1))$, consistently with our notation.

We have
\[ \sigma_m(\beta_1,\dots,\beta_n)=(-1)^m c_m(\Omega_{\bP^n}(1)),\]
and we computed before the term on the right. Set $e_i=d_i-1$, then we have
\begin{align*}
    \sigma_m(e_i\beta_1+\beta_1,\dots,e_i\beta_1+\beta_n) &= \sum_{|I|=m} \prod_{i_k\in I} (e_i\beta_1+\beta_{i_k}) \\
    &= \sum_{|I|=m} \sum_{\ell =0}^{m} e_i^\ell\beta_1^\ell\sigma_{m-\ell}(\beta_{i_1},\dots,\beta_{i_m}) \\
    &= \sum_{\ell=0}^{m} e_i^{\ell}\beta_1^\ell \left(\sum_{|I|=m} \sigma_{m-\ell}(\beta_{i_1},\dots,\beta_{i_m})\right) \\
    &= \sum_{\ell=0}^{m} e_i^{\ell}\beta_1^\ell \left(\binom{n-m+\ell}{\ell} \sigma_{m-\ell}(\beta_1,\dots,\beta_n)\right)\\
    &= \sum_{\ell=0}^{m} e_i^{\ell}\binom{n-m+\ell}{\ell}  (-1)^{m-\ell} c_{m-\ell}(\Omega_{\bP^n}(1)) \beta_1^\ell \\
    &= \sum_{\ell=0}^{m} \sum_{j=0}^{m-\ell} (-1)^{m+j-\ell} e_i^{\ell}\binom{n-m+\ell}{\ell} \beta_1^\ell \xi_1^j c_{m-\ell-j}.
\end{align*}
From this we deduce
\begin{align*}
    C_n(k_1,k_2) &= \sum_{\ell_1=0}^{n-k_1} \sum_{j_1=0}^{n-k_1-\ell_1} (-1)^{n-k_1+j_1-\ell_1} e_1^{\ell_1}\binom{k_1+\ell_1}{\ell_1} \beta_1^{\ell_1} \xi_1^{j_1} c_{n-k_1-\ell_1-j_1}\\
    &\cdot \sum_{\ell_2=0}^{n-k_2} \sum_{j_2=0}^{n-k_2-\ell_2} (-1)^{n-k_2+j_2-\ell_2} e_2^{\ell_2}\binom{k_2+\ell_2}{\ell_2} \beta_1^{\ell_2} \xi_1^{j_2} c_{n-k_2-\ell_2-j_2}.
\end{align*}
Combining these computations with \Cref{thm:chow}, we deduce the following:
\begin{proposition}\label{prop:rel cod 2}
Set $\bfd=(d_1,d_2)$. Then we have
\[\ch(\cM^{\PGL}_n(\bfd))\simeq\mathbb{Q}[\gamma_1,\gamma_2,c_2,\cdots,c_{n+1}]^{\mathfrak{S}_{\bfd}}/I \]
where $I$ is generated by the following cycles: for fixed $a$ and $b$ with $0\leq a\leq n-1$ and $0\leq b\leq n$, we have
\begin{align*}
    \sum_{k_1+k_2\leq a+b+1} &\gamma_1^{k_1}\gamma_2^{k_2} ( \sum_{\bfell+\bfj\leq n-\bfk} D(\bfk,\bfj,\bfell) c_{n-k_1-\ell_1-j_1}c_{n-k_2-\ell_2-j_2} \\
    &\cdot  \left[ s_{\ell_1+\ell_2+a-n+1}(V)s_{j_1+j_2+b-n}(V^\vee) + s_{\ell_1+\ell_2+a-n}(V)s_{j_1+j_2+b-n+1}(V^\vee) \right] ). 
\end{align*}
where $D(\bfk,\bfj,\bfell):= (-1)^{j_1+j_2-k_1-k_2-\ell_1-\ell_2}e_1^{\ell_1}e_2^{\ell_2}\binom{k_1+\ell_1}{\ell_1}\binom{k_2+\ell_2}{\ell_2}$.
\end{proposition}
Observe that the relations appearing above have degree $a+b+1$, so in particular the ideal of relations is generated in degree $d$ by $d$ relations.

\subsection{Some computations}
We already know that in degree one we have the single relation
\[ r_1= A_{1,0}\gamma_1 + A_{0,1}\gamma_2. \]
For $e_1<e_2$, we have
\begin{align*}
    &A_{1,0}=(e_2+1)\frac{e_2(e_2^n-e_1^n)+ne_1^n(e_1-e_2)}{(e_1-e_2)^2} \\
    &A_{0,1}=(e_1+1)\frac{e_1(e_1^n-e_2^n)+ne_2^n(e_2-e_1)}{(e_2-e_1)^2},
\end{align*}
whereas for $e_1=e_2=e$ we have
\[A_{1,0}=A_{0,1}=e^{n-1}(e+1)\frac{n(n+1)}{2}. \]
In degree two, we have two relations, given by computing \Cref{prop:rel cod 2} for $(a,b)=(1,0)$ and $(a,b)=(0,1)$ respectively:
\begin{align*}
  r_2^{(1,0)}=  B_{2,0}\gamma_1^2 + B_{1,1}\gamma_1\gamma_2 + B_{0,2}\gamma_2^2 + B_{0,0}c_2 \\
  r_2^{(0,1)}=  C_{2,0}\gamma_1^2 + C_{1,1}\gamma_1\gamma_2 + C_{0,2}\gamma_2^2 + C_{0,0}c_2.
\end{align*}
The term $B_{2,0}$ is given by
\begin{align*}  &\sum_{\bfell+\bfj\leq n-(2,0)} D((2,0),\bfj,\bfell) c_{n-2-\ell_1-j_1}c_{n-\ell_2-j_2} \\
    &\cdot \left[ s_{\ell_1+\ell_2-n+2}(V)s_{j_1+j_2-n}(V^\vee) + s_{\ell_1+\ell_2+1-n}(V)s_{j_1+j_2-n+1}(V^\vee) \right] 
\end{align*}
and the coefficient is non-zero only when one of the two following set of equations is satisfied
\begin{equation*}
    \begin{matrix}
    \ell_1+j_1=n-2 & & & \ell_1+j_1=n-2 \\
    \ell_2+j_2=n & & & \ell_2+j_2=n\\
    \ell_1+\ell_2=n-2 & & & \ell_1+\ell_2=n-1 \\
    j_1+j_2=n & & & j_1+j_2=n-1. \\
    \end{matrix}
\end{equation*}
After a straightforward computation we get
\begin{align*}
    B_{2,0}&=(e_2+1)\sum_{\ell_1=0}^{n-2}e_1^{\ell_1}e_2^{n-2-\ell_1}\binom{\ell_1+2}{\ell_1}\\
    &=(e_2+1)\frac{n^2e_1^{n-1}(e_1-e_2)^2 + 2e_2(e_1^n-e_2^n)-ne_1^{n-1}(e_1^2-e_2^2)}{2(e_1-e_2)} \nonumber
\end{align*}
for $e_1<e_2$, whereas for $e_1=e_2=e$ we have
\begin{align*}
    B_{2,0}=e^{n-2}(e+1)\frac{(n-1)(n+1)n}{6}.
\end{align*}
The term $B_{0,2}$ is given by 
\begin{align*}  &\sum_{\bfell+\bfj\leq n-(0,2)} D((0,2),\bfj,\bfell) c_{n-\ell_1-j_1}c_{n-2-\ell_2-j_2} \\
    &\cdot \left[ s_{\ell_1+\ell_2-n+2}(V)s_{j_1+j_2-n}(V^\vee) + s_{\ell_1+\ell_2+1-n}(V)s_{j_1+j_2-n+1}(V^\vee) \right] 
\end{align*}
and the coefficient is non-zero only when one of the two following set of equations is satisfied
\begin{equation}
    \begin{matrix}
    \ell_1+j_1=n & & & \ell_1+j_1=n \\
    \ell_2+j_2=n-2 & & & \ell_2+j_2=n-2\\
    \ell_1+\ell_2=n-2 & & & \ell_1+\ell_2=n-1 \\
    j_1+j_2=n & & & j_1+j_2=n-1. \\
    \end{matrix}
\end{equation}
After a straightforward computation we get
\begin{align*}
    B_{0,2}&=(e_1+1)\sum_{\ell_2=0}^{n-2}e_1^{n-2-\ell_2}e_2^{\ell_2}\binom{\ell_2+2}{\ell_2}\\
    &=(e_1+1)\frac{n^2e_2^{n-1}(e_1-e_2)^2 + 2e_1(e_2^n-e_1^n)-ne_2^{n-1}(e_2^2-e_1^2)}{2(e_2-e_1)}
\end{align*}
for $e_1<e_2$, and for $e_1=e_2=e$ we have $B_{2,0}=B_{0,2}$.
The term $B_{1,1}$ is given by 
\begin{align*}  &\sum_{\bfell+\bfj\leq n-(1,1)} D((1,1),\bfj,\bfell) c_{n-1-\ell_1-j_1}c_{n-1-\ell_2-j_2} \\
    &\cdot \left[ s_{\ell_1+\ell_2-n+2}(V)s_{j_1+j_2-n}(V^\vee) + s_{\ell_1+\ell_2+1-n}(V)s_{j_1+j_2-n+1}(V^\vee) \right] 
\end{align*}
and the coefficient is non-zero only when one of the two following set of equations is satisfied
\begin{equation*}
    \begin{matrix}
    \ell_1+j_1=n-1 & & & \ell_1+j_1=n-1 \\
    \ell_2+j_2=n-1 & & & \ell_2+j_2=n-1\\
    \ell_1+\ell_2=n-2 & & & \ell_1+\ell_2=n-1 \\
    j_1+j_2=n & & & j_1+j_2=n-1. \\
    \end{matrix}
\end{equation*}
After a straightforward computation we get
\begin{align*}
    B_{1,1}&=\sum_{\ell_1=0}^{n-2}e_1^{\ell_1}e_2^{n-2-\ell_1}(\ell_1+1)(n-\ell_1-1) + \sum_{\ell_1=0}^{n-1}e_1^{\ell_1}e_2^{n-1-\ell_1}(\ell_1+1)(n-\ell_1)\\
    &=(-n e_2 e_1^{n + 1} - 2 e_2 e_1^{n + 1} + n e_1^{n + 2} + n e_1 e_2^{n + 1} + 2 e_1 e_2^{n + 1} - n e_2^{n + 2})/(e_1 - e_2)^3\\
    &+ (n e_1^{n + 1} - e_1^{n + 1} - n e_2 e_1^n - e_2 e_1^n + n e_1 e_2^n + e_1 e_2^n - n e_2^{n + 1} + e_2^{n + 1})/(e_1 - e_2)^3 
\end{align*}
for $e_1<e_2$, whereas for $e_1=e_2=e$ we have
\begin{align*}
    B_{1,1}=\frac{(n+1)ne^{n-2}}{6}(n-1+e(n+2)).
\end{align*}
The terms $C_{2,0}$, $C_{0,2}$ and $C_{1,1}$ are computed in a similar way but with $a=0$ and $b=1$. The final result for $e_1<e_2$ is as follows:
\begin{align*}
    C_{2,0}&=e_2(e_2+1)\sum_{\ell_1=0}^{n-2}e_1^{\ell_1}e_2^{n-2-\ell_1}\binom{\ell_1+2}{\ell_1}\\
    &=e_2(e_2+1)\frac{n^2e_1^{n-1}(e_1-e_2)^2 + 2e_2(e_1^n-e_2^n)-ne_1^{n-1}(e_1^2-e_2^2)}{2(e_1-e_2)}=e_2B_{2,0}
\end{align*}
\begin{align*}
    C_{0,2}&=e_1(e_1+1)\sum_{\ell_2=0}^{n-2}e_1^{n-2-\ell_2}e_2^{\ell_2}\binom{\ell_2+2}{\ell_2}\\
    &=e_1(e_1+1)\frac{n^2e_2^{n-1}(e_1-e_2)^2 + 2e_1(e_2^n-e_1^n)-ne_2^{n-1}(e_2^2-e_1^2)}{2(e_2-e_1)}=e_1B_{0,2}
\end{align*}
\begin{align*}
    C_{1,1}&=e_1\sum_{\ell_1=0}^{n-2}e_1^{\ell_1}e_2^{n-2-\ell_1}(\ell_1+2)(n-\ell_1) + \sum_{\ell_1=0}^{n-1}e_1^{\ell_1}e_2^{n-1-\ell_1}(\ell_1+1)(n-\ell_1) \\
    &=\frac{-3 e_2 e_1^{n + 1} + e_2^2 e_1^n - e_1^3 e_2^{n-1} + 3 e_1^2 e_2^n - n (e_1 - e_2) (-2e_1^{n + 1} + e_2 e_1^n + e_1^2 e_2^{n-1} - 2 e_1 e_2^{n })}{(e_1 - e_2)^3} \\
    &+ \frac{n e_1^{n + 1} - e_1^{n + 1} - n e_2 e_1^n - e_2 e_1^n + n e_1 e_2^n + e_1 e_2^n - n e_2^{n + 1} + e_2^{n + 1}}{(e_1 - e_2)^3 }.
\end{align*}
For $e_1=e_2=e$, we have
\begin{align*}
    &C_{2,0}=eB_{2,0}=eB_{0,2}=C_{0,2} \\
    &C_{1,1}=\frac{e^{n-1}}{6}((n+1)(2n^2+7n-6)).
\end{align*}

Let us compute the coefficient in front of $c_{d+1}$ in the relation \Cref{prop:rel cod 2} for $0\leq b\leq d$ (hence $a=d-b$). There are eight set of equations whose resulting values for $\bfell$ and $\bfj$ contribute to $c_{d+1}$. The first four sets are
\begin{align*}
    &\begin{cases}
    \ell_1+j_1=n-d-1\\
    \ell_2+j_2=n\\
    \ell_1+\ell_2=n-a-1\\
    j_1+j_2=n-b
    \end{cases}
    &&\begin{cases}
     \ell_1+j_1=n-d-1 \\
     \ell_2+j_2=n \\
     \ell_1+\ell_2=n-a \\
     j_1+j_2=n-b-1
    \end{cases}\\
    &\begin{cases}
     \ell_1+j_1=n \\
     \ell_2+j_2=n-d-1  \\
     \ell_1+\ell_2=n-a-1\\
     j_1+j_2=n-b
    \end{cases}
    &&\begin{cases}
    \ell_1+j_1=n \\
    \ell_2+j_2=n-d-1 \\
    \ell_1+\ell_2=n-a \\
    j_1+j_2=n-b-1.
    \end{cases}
\end{align*}
The four other contributions come from the Segre classes $s_{d+1}(V)=-c_{d+1}$ and $s_{d+1}(V^{\vee})=(-1)^dc_{d+1}$ appearing in \Cref{prop:rel cod 2}. The possible values for $\bfell$ and $\bfj$ are the ones that satisfy one of these four systems
\begin{align*}
    &\begin{cases}
    \ell_1+j_1=n \\
    \ell_2+j_2=n   \\
    \ell_1+\ell_2=n-a+d  \\
    j_1+j_2=n-b
    \end{cases}
    &&\begin{cases}
    \ell_1+j_1=n \\
    \ell_2+j_2=n  \\
    \ell_1+\ell_2=n-a-1 \\
    j_1+j_2=n+d-b+1
    \end{cases}\\
    &\begin{cases}
     \ell_1+j_1=n\\
     \ell_2+j_2=n\\
     \ell_1+\ell_2=n+d-a+1\\
      j_1+j_2=n-b-1
    \end{cases}
    &&\begin{cases}
     \ell_1+j_1=n \\
     \ell_2+j_2=n \\
     \ell_1+\ell_2=n-a \\
     j_1+j_2=n+d-b.
    \end{cases}
\end{align*}
Putting all together, we get that
\begin{align}\label{eq:C(d,b)}
    C(d,b)_{0,0}=&(-1)^{d-1}(e_2^b(e_2+1)+e_1^b(e_1+1)) \sum_{\ell_1=0}^{n-d-1}e_1^{\ell_1}e_2^{n-d-1-\ell_1} \nonumber\\
    &- e_1^be_2^b \sum_{\ell_1=0}^{n-b}e_1^{\ell_1}e_2^{n-b-\ell_1}+ (-1)^{d}\sum_{\ell_1=0}^{n-d+b-1}e_1^{\ell_1}e_2^{n-d+b-1-\ell_1}\\
    &- e_1^{b+1}e_2^{b+1}\sum_{\ell_1=0}^{n-b-1}e_1^{\ell_1}e_2^{n-b-1-\ell_1} + (-1)^{d}\sum_{\ell_1=0}^{n-d+b}e_1^{\ell_1}e_2^{n-d+b-\ell_1} \nonumber.
\end{align}
Assuming $e_1<e_2$, after some simplifications, we get
\begin{equation}\label{eq:C(d,b) first}
    C(d,b)_{0,0} = \frac{e_1^b(e_1+1)e_2^{n-d}(e_2^{d+1}+(-1)^d) - e_2^b(e_2+1)e_1^{n-d}(e_1^{d+1}+(-1)^d)}{e_1-e_2}.
\end{equation}
Assuming $e_1=e_2=e$, then after further simplifications we can rewrite (\ref{eq:C(d,b)}) as follows:
\begin{align}\label{eq:C(d,b) second}
   C(d,b)_{0,0} =  e^{n-d+b}(&be(e+1)(e^d+(-1)^d)-(n(e+1)+1)e^{d+1} \nonumber \\
   &+(-1)^{d-1}((n-d)(e+1)+2n-1)e +(-1)^d2(n-d)). 
\end{align}
In particular, with the computations we have done so far we are able to write down an explicit formula for two quantities which will be relevant for the main result of this Section.

Set $B_{0,0}:=C(1,0)_{0,0}$ and $C_{0,0}:=C(1,1)_{0,0}$. The first polynomial we consider is 
\begin{align}\label{eq:det imp}
    (A_{0,1}^2B_{2,0} - A_{1,0}A_{0,1}B_{1,1} + A_{1,0}^2B02)C_{0,0} \\
    - (A_{0,1}^2C_{2,0} - A_{1,0}A_{0,1}C_{1,1} + A_{1,0}^2C_{0,2})B_{0,0}.\nonumber
\end{align}
when $e_1<e_2$. Using the formulas we determined so far, we deduce an explicit expression for (\ref{eq:det imp}), which is
\begin{multline*}
    (1/(2 (e_1 - e_2)^8))(1 + e_1) (1 + e_2) ((1/
   e_2)(e_1^{-1 + n} (-1 + e_1^2) (1 + e_2) \\
   - (1 + e_1) e_2^{-1 + 
        n} (-1 + e_2^2)) ((1/
      e_1)(1 + e_1) (e_1 - e_2)^2 e_2^2 (e_1^{1 + n}\\
      + e_2^{1 + n} n - 
         e_1 e_2^n (1 + n))^2 (-2 e_1 e_2^{1 + n} + 
         e_1^{2 + n} (-1 + n) n + \\
         e_1^n e_2^2 n (1 + n) - 
         2 e_1^{1 + n} e_2 (-1 + n^2)) \\
         -e_1 (e_1 - e_2)^2 (1 + e_2) (e_2^{1 + n} + e_1^{1 + n} n - 
         e_1^n e_2 (1 + n))^2 (-2 e_1^{1 + n} e_2\\
         + e_2^{2 + n} (-1 + n) n + e_1^2 e_2^n n (1 + n) - 
         2 e_1 e_2^{1 + n} (-1 + n^2)) \\
         - 2 (e_2^{1 + n} + e_1^{1 + n} n - e_1^n e_2 (1 + n)) (e_1^{1 + n} + 
         e_2^{1 + n} n - e_1 e_2^n (1 + n)) (-e_2^{2 + n} (-1 + n)\\
         + 2 e_1^{2 + n} e_2 n + e_1^n (-1 + e_2) e_2^2 (1 + n) - 
         e_1^3 e_2^n (1 + n) + 3 e_1^2 e_2^{1 + n} (1 + n)\\
         -e_1 e_2^{1 + n} (-1 + (-1 + 2 e_2) n) - 
         e_1^{1 + n} e_2 (1 - n + 3 e_2 (1 + n)))) - (e_1^{-1 + 
        n} (-1 + e_1^2) e_2 (1 + e_2) \\
        - e_1 (1 + e_1) e_2^{-1 + n} (-1 + e_2^2)) ((1/
      e_1)(1 + e_1) (e_1 - e_2)^2 (e_1^{1 + n} + e_2^{1 + n} n \\
      -e_1 e_2^n (1 + n))^2 (-2 e_1 e_2^{1 + n} + 
         e_1^{2 + n} (-1 + n) n + e_1^n e_2^2 n (1 + n) \\
         -2 e_1^{1 + n} e_2 (-1 + n^2)) - (1/
      e_2)(e_1 - e_2)^2 (1 + e_2) (e_2^{1 + n} + e_1^{1 + n} n - 
         e_1^n e_2 (1 + n))^2 (-2 e_1^{1 + n} e_2 \\
         +e_2^{2 + n} (-1 + n) n + e_1^2 e_2^n n (1 + n) - 
         2 e_1 e_2^{1 + n} (-1 + n^2)) \\
         + 2 (e_2^{1 + n} + e_1^{1 + n} n - e_1^n e_2 (1 + n)) (e_1^{1 + n} \\
      +  e_2^{1 + n} n - e_1 e_2^n (1 + n)) (-e_1^{2 + n} n + 
         e_1^n e_2 (1 + n) + e_2^{1 + n} (-1 + n + e_2 n) \\
         + e_1^{1 + n} (1 - n + e_2 (2 + n)) - 
         e_1 e_2^n (1 + n + e_2 (2 + n)))))
\end{multline*}
The second polynomial is 
\begin{align}\label{eq:det simple}
    (B_{1,1}-2B_{2,0})C_{0,0}-(C_{1,1}-2C_{2,0})B_{0,0}
\end{align}
when $e_1=e_2=e$. After some computations, we get that this is equal to
\begin{multline*}
    \frac{e^{2n-2}}{6}(e^4 n (4 + n - 4 n^2 - n^3) \\
    + e^3 n (-9 - 4 n + 6 n^2 + n^3) + 
 e (-12 + 25 n + 14 n^2 - 28 n^3 - 5 n^4) \\
 +  2 (6 - 8 n - 7 n^2 + 8 n^3 + n^4) + 
 e^2 (-12 - 4 n + 15 n^2 + 10 n^3 + 3 n^4))
\end{multline*}
\subsection{Main results}
We are ready to state the main results of the Section.
\begin{theorem}\label{thm:chow cod 2 first}
Let $n\geq 3$ and $d_1>d_2\geq 2$ be integers such that the quantity (\ref{eq:det imp}) for $e_i=d_i-1$ is not zero. Then
\[\ch(\cM^{\PGL}_n(\bfd))\simeq \mathbb{Q}[\gamma_1]/(\gamma_1^2). \]
\end{theorem}
\begin{proof}
We know from \Cref{prop:rel cod 2} that in this case we have
\[ \ch(\cM^{\PGL}_n(\bfd))\simeq \mathbb{Q}[\gamma_1,\gamma_2,c_2,c_3,\dots,c_{n+1}]/I. \]
In degree $1$ we have the single relation
\[ r_1= A_{1,0}\gamma_1 + A_{0,1}\gamma_2. \]
For $d_1<d_2$, we deduce from the relation $r_1$ that $\gamma_2=-(A_{1,0}/A_{0,1})\gamma_1$.
In degree $2$ we have the two relations
\begin{align*}
  r_2^{(1,0)}=  B_{2,0}\gamma_1^2 + B_{1,1}\gamma_1\gamma_2 + B_{0,2}\gamma_2^2 + B_{0,0}c_2 \\
  r_2^{(0,1)}=  C_{2,0}\gamma_1^2 + C_{1,1}\gamma_1\gamma_2 + C_{0,2}\gamma_2^2 + C_{0,0}c_2.
\end{align*}
Substituting $\gamma_2=-(A_{1,0}/A_{0,1})\gamma_1$, we get the following system of equations:
\begin{equation}\label{eq:fund matrix 1}
\left(\begin{matrix}
  B_{2,0}-\frac{A_{1,0}}{A_{0,1}}B_{1,1}+\frac{A_{1,0}^2}{A_{0,1}^2}B_{0,2} & B_{0,0} \\
  C_{2,0}-\frac{A_{1,0}}{A_{0,1}}C_{1,1}+\frac{A_{1,0}^2}{A_{0,1}^2}C_{0,2} & C_{0,0}
\end{matrix}\right)
\left(\begin{matrix}
  \gamma_1^2 \\
  c_2
\end{matrix}\right)
= 0.
\end{equation}
If the determinant of the matrix appearing in (\ref{eq:fund matrix 1}) is non-zero, we have $\gamma_1^2=c_2=0$.
It's straightforward to check that this condition is equivalent to the quantity (\ref{eq:det imp}) being non-zero.
We are left with proving that $c_i=0$ for $i\geq 3$.

Using the explicit expression obtained in (\ref{eq:C(d,b) first}), for $e_1<e_2$ we have that the solutions to $C(d,b)_{0,0}=0$, regarded as an equation in one variable $b$, are the same as the solutions to the equation
\begin{align*}
     e_1^b(e_1+1)e_2^{n-d}(e_2^{d+1}+(-1)^d) - e_2^b(e_2+1)e_1^{n-d}(e_1^{d+1}+(-1)^d)=e_1^b E - e_2^b F=0.
\end{align*}
As $e_2>e_1\geq 1$, we have that $E\neq 0$, hence $(e_1/e_2)^b=F/E$. In particular, there is at most one integer $b$ for which this equation is satisfied.

Now we prove by induction that $c_i=0$ for $i\geq 3$, the first case being $i=3$: \Cref{prop:rel cod 2} combined with the fact that $\gamma_1^{k_1}\gamma_2^{k_2}=0$ for $k_1+k_2\geq 3$ implies that we have three relations
\[ C(2,0)_{0,0}c_3=C(2,1)_{0,0}c_3=C(2,2)_{0,0}c_3=0. \]
We have just seen that there is at most one value of $b$ for which $C(2,b)_{0,0}=0$: this immediately implies that $c_3=0$ in the rational Chow ring.

The inductive step proceeds along the same lines: assuming that $c_i=0$ for $i=1,2,\cdots, d$, using again the fact that $\gamma_1^{k_1}\gamma_2^{k_2}=0$ for $k_1+k_2\geq 3$, we get that the relations given in \Cref{prop:rel cod 2} are 
\[ C(d,b)_{0,0}c_{d+1}=0,\quad b=0,\dots,d. \]
The same argument used before shows that there is at most one value of $b$ for which $C(d,b)_{0,0}=0$, which readily implies that $c_{d+1}=0$.
\end{proof}

\begin{theorem}\label{thm:chow cod 2 second}
Let $n\geq 3$ and $d_1=d_2\geq 2$ be integers such that the quantity (\ref{eq:det simple}) for $e=d_1-1$ is not zero. Then
\[\ch(\cM^{\PGL}_n(\bfd))\simeq \mathbb{Q}. \]
\end{theorem}
\begin{proof}
We know from \Cref{prop:rel cod 2} that for $d_1=d_2$ we have
\[ \ch(\cM^{\PGL}_n(\bfd))\simeq \mathbb{Q}[\gamma_1+\gamma_2,\gamma_1\gamma_2,c_2,c_3,\dots,c_{n+1}]/I. \]
If $d_1=d_2$, the only relation in degree one is
\[ A_{1,0}\gamma_1+A_{0,1}\gamma_2=A_{1,0}(\gamma_1+\gamma_2)=0 \]
which implies $\gamma_1+\gamma_2=0$. In degree two, using the fact that $B_{1,0}=B_{0,1}$ and $C_{1,0}=C_{0,1}$ we have
\begin{align*}
  r_2^{(1,0)}=  B_{2,0}(\gamma_1^2+\gamma_2^2) + B_{1,1}\gamma_1\gamma_2 + B_{0,0}c_2 \\
  r_2^{(0,1)}=  C_{2,0}(\gamma_1^2 + \gamma_2^2) + C_{1,1}\gamma_1\gamma_2 + C_{0,0}c_2.
\end{align*}
We can rewrite $\gamma_1^2+\gamma_2^2$ as $(\gamma_1+\gamma_2)^2-2\gamma_1\gamma_2$, hence we deduce the following two relations:
\begin{equation}\label{eq:fund matrix 2}
\left(\begin{matrix}
  B_{1,1}-2B_{2,0} & B_{0,0} \\
  C_{1,1}-2C_{2,0} & C_{0,0}
\end{matrix}\right)
\left(\begin{matrix}
  \sigma_2(\gamma_1,\gamma_2) \\
  c_2
\end{matrix}\right)
= 0.
\end{equation}
If the determinant of the matrix in (\ref{eq:fund matrix 2}) is non-zero, we deduce that $\gamma_1\gamma_2=c_2=0$. This condition is equivalent to the quantity (\ref{eq:det simple}) being non-zero.

To show that $c_d=0$ for $d\geq 3$, we use the same argument of the proof of \Cref{thm:chow cod 2 first}. In this case, from (\ref{eq:C(d,b) second}) we have that $C(d,b)_{0,0}=0$ if and only if 
\begin{align*}
    be(e+1)(e^d+(-1)^d)-(n(e+1)+1)e^{d+1}\\
    +(-1)^{d-1}((n-d)(e+1)+2n-1)e +(-1)^d2(n-d))=0.
\end{align*}
If $d$ is even or $e\neq 1$, then again there is at most one $b$ which solves the equation above; otherwise, for $d$ odd and $e=1$, it's straightforward to check that the expression above is non-zero. Then the induction argument used in the proof of \Cref{thm:chow cod 2 first} applies also here.
\end{proof}
\subsection{Some applications}
We give two immediate applications of \Cref{thm:chow cod 2 first}. In the first one, we reprove a result of Faber.
\begin{corollary}[\cite{Fab-m4}]\label{cor:chow M4}
Let $M_4$ be the moduli space of smooth curves of genus four. Then
\[ \ch(M_4)\simeq \mathbb{Q}[\lambda_1]/(\lambda_1^3) \]
where $\lambda_1$ is the first Chern class of the Hodge bundle.
\end{corollary}
\begin{proof}
Let $H_4$ be the moduli space of hyperelliptic curves of genus four, regarded as a subvariety of $M_4$. We observed in \Cref{rmk:examples} that $M_4\smallsetminus H_4$ is isomorphic to the coarse moduli space of $\cM^{\PGL}_3(2,3)$, hence it follows from \Cref{thm:chow cod 2 first} that 
\[ \ch(M_4\smallsetminus H_4)\simeq \mathbb{Q}[\gamma_1]/(\gamma_1^2). \]
In particular, as ${\rm CH}^1(M_4)\simeq {\rm CH}^1(M_4\smallsetminus H_4)$, the Hodge class $\lambda_1$ must be a multiple of $\gamma_1$.

The Chow ring of $H_4$ is trivial, the latter being an open subvariety of $\AA^{2g-1}$, and the fundamental class of $H_4$ is equal to a multiple of $\lambda_1^3$. These two facts, combined with the localization exact sequence
\[ {\rm CH}^{*-2}(H_4) \longrightarrow \ch(M_4) \longrightarrow \ch(M_4\smallsetminus H_4) \longrightarrow 0\]
tell us that the Chow ring of $M_4$ is isomorphic to $\mathbb{Q}[\lambda_1]/(\lambda_1^i)$ where $i$ is either $2$ or $3$. As we know from \cite{Fab-conj}*{Theorem 2} that $\lambda_1^2$ is not zero, we get the claimed result.
\end{proof}

The second application concerns the coarse moduli space $K_6$ of polarized K3 surfaces of degree six. Here we adopt the same notation of \Cref{prop:chow U8}, where we denoted the Noether-Lefschetz divisors by $D_{d,h}$.
\begin{corollary}\label{cor:chow U6}
Let $U_6\subset K_6$ be the open subvariety parametrizing polarized K3 surfaces of degree six whose polarization is very ample. Then
\[ \ch(U_6)\simeq \mathbb{Q}[\lambda_1]/(\lambda_1^2), \]
where $\lambda_1$ is the Hodge line bundle and the pushforward morphism
\[ {\rm CH}^{i-1}(\cup_{d=1}^{3} D_{d,1})\longrightarrow{\rm CH}^i(K_8) \]
is surjective for $i>1$.
\end{corollary}
\begin{proof}
As observed in \Cref{rmk:examples}, the coarse space of $\cM^{\PGL}_4(2,3)$ is isomorphic to $U_6$. We can then apply \Cref{thm:chow cod 2 first} with $d_1=2$, $d_2=3$ and $n=4$. The fact that $\gamma_1$ is a non-zero multiple of $\lambda_1$ follows from \cite{DL-k3}*{Proposition 4.2.6}, and the claim on the pushforward morphism follows from $U_6$ being the complement of the union of those Noether-Lefschetz divisors in $K_6$.
\end{proof}

\appendix
\section{Quotient bundles and Grassmannians}\label{sec:quot}
Let $V$ be a vector space of dimension $n$ and let $W\subset V$ be a vector subspace of dimension $m$. For $1\leq r\leq (n-m)$, consider the two Grassmannians $\Gr(r,V)$ and $\Gr(r,V/W)$. Let $U\subset\Gr(r,V)$ be the open subscheme whose points $[E]$ corresponds to $r$-planes $E\subset V$ such that $E\cap W=\{0\}$ (observe that the numerical condition on $r$ implies that $U$ is not empty). There exists a well defined map
\[ q: U\longrightarrow \Gr(r,V/W),\quad [E]\longmapsto [\overline{E}] \]
where $[\overline{E}]$ is the image of $E$ in the quotient vector space $V/W$ (the rank of $[\overline{E}]$ is still $r$ because $E\cap W=\{0\}$).
\begin{proposition}\label{prop:affine}
The map $U\rightarrow \Gr(r,V/W)$ defined above is an affine bundle. 
\end{proposition}
\begin{proof}
For this, let us look at the fiber $q^{-1}([\overline{E}])$ over a point $[\overline{E}]$: this consists of all the $r$-planes in $V$ whose image in the quotient vector space coincides with the one of $E$. Let us fix a basis $\{e_1,\dots , e_{r} \}$ for $E$ and a basis $\{ f_1,\dots, f_m \}$ for $W$. There is a map
\begin{equation}\label{eq:affine bundle}
   \mathbb{A}^{mr}\simeq {\rm{Mat}}_{m,r} \longrightarrow q^{-1}([\overline{E}]) 
\end{equation}  
given by 
\begin{equation}\label{eq:matrix map}
    A\longmapsto \begin{pNiceArray}{c|c|c}
  &  &   \\
 e_1 & \hdots & e_r\\
  &  & 
\end{pNiceArray} + \begin{pNiceArray}{c|c|c}
  &  &   \\
 f_1 & \hdots & f_m \\
  &  & 
\end{pNiceArray} A
\end{equation}
where the matrix in the right hand side should be interpreted as the linear subspace spanned by the column vectors. Observe that the condition $E\cap W=\{0\}$ implies that the image of (\ref{eq:affine bundle}) is indeed in $U$.

We claim that (\ref{eq:affine bundle}) is an isomorphism. To prove that it is surjective, observe that given a point $[E']$ in the fiber and a basis $e'_1,\dots, e'_r$ for the associated subspace $E'$, then there must exists an $r\times r$-matrix $C$ and a matrix $A$ such that
\[ \begin{pNiceArray}{c|c|c}
  &  &   \\
 e'_1 & \hdots & e'_r\\
  &  & 
\end{pNiceArray}=\begin{pNiceArray}{c|c|c}
  &  &   \\
 e_1 & \hdots & e_r\\
  &  & 
\end{pNiceArray} C + \begin{pNiceArray}{c|c|c}
  &  &   \\
 f_1 & \hdots & f_m \\
  &  & 
\end{pNiceArray} A. \]
If we multiply on the right by $C^{-1}$, we get
\[ 
\begin{pNiceArray}{c|c|c}
  &  &   \\
 e''_1 & \hdots & e''_r\\
  &  & 
\end{pNiceArray} = 
\begin{pNiceArray}{c|c|c}
  &  &   \\
 e'_1 & \hdots & e'_r\\
  &  & 
\end{pNiceArray}C^{-1}=\begin{pNiceArray}{c|c|c}
  &  &   \\
 e_1 & \hdots & e_r\\
  &  & 
\end{pNiceArray} + \begin{pNiceArray}{c|c|c}
  &  &   \\
 f_1 & \hdots & f_m \\
  &  & 
\end{pNiceArray} (AC^{-1}), \]
which means that $AC^{-1}\mapsto [E']$. This proves surjectivity.

Suppose now that there exist two different basis $e_1',\dots, e_r'$ and $e_1'',\dots, e_r''$ for the same subspace $E'$ of the form
\[
\begin{pNiceArray}{c|c|c}
  &  &   \\
 e'_1 & \hdots & e'_r\\
  &  & 
\end{pNiceArray}=
\begin{pNiceArray}{c|c|c}
  &  &   \\
 e_1 & \hdots & e_r\\
  &  & 
\end{pNiceArray} + \begin{pNiceArray}{c|c|c}
  &  &   \\
 f_1 & \hdots & f_m \\
  &  & 
\end{pNiceArray} A',
\]
\[
\begin{pNiceArray}{c|c|c}
  &  &   \\
 e''_1 & \hdots & e''_r\\
  &  & 
\end{pNiceArray}=
\begin{pNiceArray}{c|c|c}
  &  &   \\
 e_1 & \hdots & e_r\\
  &  & 
\end{pNiceArray} + \begin{pNiceArray}{c|c|c}
  &  &   \\
 f_1 & \hdots & f_m \\
  &  & 
\end{pNiceArray} A''.
\]
As both $e_1',\dots, e_r'$ and $e_1'',\dots, e_r''$ span the same vector subspace, there exists an invertible matrix $C$ of rank $r$ such that 
\[ \begin{pNiceArray}{c|c|c}
  &  &   \\
 e'_1 & \hdots & e'_r\\
  &  & 
\end{pNiceArray} =
\begin{pNiceArray}{c|c|c}
  &  &   \\
 e''_1 & \hdots & e''_r\\
  &  & 
\end{pNiceArray}C. \]
This readily implies that
\[
\begin{pNiceArray}{c|c|c}
  &  &   \\
 e_1 & \hdots & e_r\\
  &  & 
\end{pNiceArray}({\rm{Id}}-C)=\begin{pNiceArray}{c|c|c}
  &  &   \\
 f_1 & \hdots & f_m \\
  &  & 
\end{pNiceArray} (A''C-A').
\]
Observe that the left hand side belongs to $E$ whereas the right hand side belongs to $W$. As $E\cap W=\{0\}$, we deduce that $C={\rm{Id}}$, hence (\ref{eq:affine bundle}) is injective. This easily implies that $q^{-1}(\overline{E})\simeq\mathbb{A}^{rm}$ and that \[ q: U\longrightarrow \Gr(r,V/W),\quad [E]\longmapsto [\overline{E}] \] is an affine bundle.
\end{proof}


\begin{bibdiv}
	\begin{biblist}
	
   \bib{ACGH}{book}{
   author={Arbarello, E.},
   author={Cornalba, M.},
   author={Griffiths, P. A.},
   author={Harris, J.},
   title={Geometry of algebraic curves. Vol. I},
   series={Grundlehren der mathematischen Wissenschaften [Fundamental
   Principles of Mathematical Sciences]},
   volume={267},
   publisher={Springer-Verlag, New York},
   date={1985},
   }
  
  \bib{AI}{article}{
   author={Asgarli, Shamil},
   author={Inchiostro, Giovanni},
   title={The Picard group of the moduli of smooth complete intersections of
   two quadrics},
   journal={Trans. Amer. Math. Soc.},
   volume={372},
   date={2019},
   number={5},
}
   
   \bib{benoist-thesis}{thesis}{
	author={Benoist, Olivier},
	title={Espaces de modules d'intersections compl\'{e}tes lisses},
	type={Ph.D. Thesis},
	organization={Universit\`{e} Paris Diderot (Paris 7)},
	date={2012},
	}
	
\bib{Ben-deg}{article}{
   author={Benoist, Olivier},
   title={Degr\'{e}s d'homog\'{e}n\'{e}it\'{e} de l'ensemble des intersections compl\`etes
   singuli\`eres},
   language={French, with English and French summaries},
   journal={Ann. Inst. Fourier (Grenoble)},
   volume={62},
   date={2012},
   number={3},
}

\bib{Ben-sep}{article}{
   author={Benoist, Olivier},
   title={S\'{e}paration et propri\'{e}t\'{e} de Deligne-Mumford des champs de modules
   d'intersections compl\`etes lisses},
   language={French, with English and French summaries},
   journal={J. Lond. Math. Soc. (2)},
   volume={87},
   date={2013},
   number={1},
}


\bib{CL}{article}{			
			author={Canning, Samir},
			author={Larson, Hannah},
			title={The {C}how rings of the moduli spaces of curves of genus 7, 8 and 9},
			status={preprint},
			eprint={https://arxiv.org/abs/2104.05820},
			date={2021},
			} 


\bib{DL-coh}{article}{
   author={Di Lorenzo, Andrea},
   title={Cohomological invariants of the stack of hyperelliptic curves of
   odd genus},
   journal={Transform. Groups},
   volume={26},
   date={2021},
   number={1},
}
   
\bib{DL-pic-curves}{article}{
   author={Di Lorenzo, Andrea},
   title={Picard group of moduli of curves of low genus in positive
   characteristic},
   journal={Manuscripta Math.},
   volume={165},
   date={2021},
   number={3-4},
}

\bib{DL-k3}{article}{			
			author={Di Lorenzo, Andrea},
			title={Integral Picard group of the stack of quasi-polarized K3 surfaces of low degree},
			status={preprint},
			eprint={https://arxiv.org/abs/1910.08758},
			date={2019},
			}
\bib{DLFV}{article}{
   author={Di Lorenzo, Andrea},
   author={Fulghesu, Damiano},
   author={Vistoli, Angelo},
   title={The integral Chow ring of the stack of smooth non-hyperelliptic
   curves of genus three},
   journal={Trans. Amer. Math. Soc.},
   volume={374},
   date={2021},
   number={8},
}	

\bib{EG}{article}{
   author={Edidin, Dan},
   author={Graham, William},
   title={Equivariant intersection theory},
   journal={Invent. Math.},
   volume={131},
   date={1998},
   number={3},
}

\bib{EG-loc}{article}{
   author={Edidin, Dan},
   author={Graham, William},
   title={Localization in equivariant intersection theory and the Bott
   residue formula},
   journal={Amer. J. Math.},
   volume={120},
   date={1998},
   number={3},
}

\bib{Fab-conj}{article}{
   author={Faber, Carel},
   title={A conjectural description of the tautological ring of the moduli
   space of curves},
   conference={
      title={Moduli of curves and abelian varieties},
   },
   book={
      series={Aspects Math., E33},
      publisher={Friedr. Vieweg, Braunschweig},
   },
   date={1999},
   pages={109--129},
   review={\MR{1722541}},
}

\bib{Fab-m3}{article}{
   author={Faber, Carel},
   title={Chow rings of moduli spaces of curves. I. The Chow ring of
   $\overline{\scr M}_3$},
   journal={Ann. of Math. (2)},
   volume={132},
   date={1990},
   number={2},
}

\bib{Fab-m4}{article}{
   author={Faber, Carel},
   title={Chow rings of moduli spaces of curves. II. Some results on the
   Chow ring of $\overline{\scr M}_4$},
   journal={Ann. of Math. (2)},
   volume={132},
   date={1990},
   number={3},
}

\bib{FVis}{article}{
   author={Fulghesu, Damiano},
   author={Vistoli, Angelo},
   title={The Chow ring of the stack of smooth plane cubics},
   journal={Michigan Math. J.},
   volume={67},
   date={2018},
   number={1},
}


\bib{FV}{article}{
   author={Fulghesu, Damiano},
   author={Viviani, Filippo},
   title={The Chow ring of the stack of cyclic covers of the projective
   line},
   language={English, with English and French summaries},
   journal={Ann. Inst. Fourier (Grenoble)},
   volume={61},
   date={2011},
   number={6},
}

\bib{Iza}{article}{
   author={Izadi, E.},
   title={The Chow ring of the moduli space of curves of genus $5$},
   conference={
      title={The moduli space of curves},
      address={Texel Island},
      date={1994},
   },
   book={
      series={Progr. Math.},
      volume={129},
      publisher={Birkh\"{a}user Boston, Boston, MA},
   },
   date={1995},
}


\bib{PV}{article}{
author={Patel, A.},
author={Vakil, R.},
title={On the Chow ring of the Hurwitz space of degree three covers of $\PP^1$},
status={preprint},
date={2015},
eprint={https://arxiv.org/abs/1505.04323},
}
	

\bib{PeV}{article}{
   author={Penev, Nikola},
   author={Vakil, Ravi},
   title={The Chow ring of the moduli space of curves of genus six},
   journal={Algebr. Geom.},
   volume={2},
   date={2015},
   number={1},
}
		
	\end{biblist}
\end{bibdiv}
\end{document}